\documentclass[11pt, twoside,a4paper]{amsart}
\usepackage{float}
\makeatletter
\let\@wraptoccontribs\wraptoccontribs
\makeatother

\title[Algebras from NIM-Reps]{Reconstructing algebra objects from NIM-reps in pointed, near-group and quantum group-like fusion categories}

\date{\today}

\author{Samuel Hannah}
\address{School of Mathematics, Cardiff University, Abacws, Senghennydd Road, Cardiff, CF24 4AG, Wales}
\email{hannahs@cardiff.ac.uk}

\author{Ana Ros Camacho}
\address{School of Mathematics, Cardiff University, Abacws, Senghennydd Road, Cardiff, CF24 4AG, Wales}
\email{roscamachoa@cardiff.ac.uk}

\contrib[with an appendix with]{Devi Young}

\usepackage{mathabx}
\usepackage{import}
\usepackage{amsmath}
\usepackage{amsfonts}
\usepackage{amsthm}
\usepackage{amssymb,bbm}
\usepackage{dsfont}
\usepackage[alphabetic, initials, nobysame]{amsrefs}
\usepackage[english]{babel}
\usepackage{url}
\usepackage{fancyhdr}
\usepackage{graphicx}
\usepackage{verbatim}
\usepackage[normalem]{ulem}
\usepackage[shortlabels]{enumitem}
\newcommand{\stkout}[1]{\ifmmode\text{\sout{\ensuremath{#1}}}\else\sout{#1}\fi}
\usepackage[
colorlinks=true,
linkcolor=black, 
anchorcolor=black,
citecolor=black,
urlcolor=black, 
]{hyperref}
\usepackage[all]{xy}
\usepackage{geometry}
\usepackage{microtype}
\usepackage[dvipsnames]{xcolor}
\usepackage{cleveref}
\usepackage{tikz-cd}

\allowdisplaybreaks

\usepackage{calrsfs}
\DeclareMathAlphabet{\cal}{OMS}{zplm}{m}{n}

\input xy
\xyoption{all}

\usepackage{enumitem}

\usepackage[justification=centering]{caption}

\newcommand{\coev}{\mathsf{coev}}
\newcommand{\coevr}{\widetilde{\mathsf{coev}}}

\newcommand{\ev}{\mathsf{ev}}
\newcommand{\evr}{\widetilde{\mathsf{ev}}}

\newcommand{\Hom}{\mathsf{Hom}}

\newcommand{\ide}{\mathsf{Id}}

\newcommand{\Ob}{\mathsf{Ob}}
\newcommand{\one}{\mathds{1}}

\newcommand{\rev}{\mathrm{rev}}

\newcommand{\Mod}{\mathsf{Mod}}

\newcommand{\act}{\vartriangleright}

\newcommand{\mZ}{\mathbb{Z}}

\newcommand{\mZZ}{\mathbb{Z}_{\geq 0}}

\newcommand{\cC}{\cal{C}}

\newcommand{\cM}{\cal{M}}

\newcommand{\cZ}{\cal{Z}}

\newtheoremstyle{defstyle}
  {0.5cm} 
  {0.5cm}
  {\normalfont}
  {}  
  {\normalfont\bfseries}
  {:}
  {0.3cm}
  {\thmname{#1}\thmnumber{ #2}\thmnote{ (#3)}}

\numberwithin{equation}{subsection}

\newtheorem*{rep@theorem}{\rep@title}
\newcommand{\newreptheorem}[2]{
\newenvironment{rep#1}[1]{
 \def\rep@title{#2 \ref{##1}}
 \begin{rep@theorem}}
 {\end{rep@theorem}}}
\makeatother

\newtheorem{theorem}{Theorem}[section]

\newtheorem{proposition}[theorem]{Proposition}
\newreptheorem{proposition}{Proposition}
\newtheorem{corollary}[theorem]{Corollary}
\newreptheorem{corollary}{Corollary}
\newtheorem{lemma}[theorem]{Lemma}
\newtheorem{conjecture}[theorem]{Conjecture}

\newtheorem{theorem*}{Theorem}
\newreptheorem{theorem}{Theorem}

\theoremstyle{definition}
\newtheorem{definition}[theorem]{Definition}
\newtheorem{notation}[theorem]{Notation}

\newtheorem{example}[theorem]{Example}
\newtheorem{remark}[theorem]{Remark}

\begin{document}

\maketitle

\begin{abstract}
In this article we study the possible Morita equivalence classes of algebras in three families of fusion categories (pointed, near-group and $A \left( 1,l \right)_{\frac{1}{2}}$) by studying the Non-negative Integer Matrix representations (NIM-reps) of their underlying fusion ring, and compare these results with existing classification results of algebra objects. Also, in an appendix we include a test of the exponents conjecture for modular tensor categories of rank up to 4. 
\end{abstract}


{\footnotesize \tableofcontents}

\section{Introduction}

The study of fusion rings and their associated non-integer matrix representations (or NIM-reps, for short) has stimulated a rich production in the literature, see e.g. \cites{Behrend, DiFZ,EP,GannonNIM,Gannon}. A good reason behind this is the connection of NIM-reps to boundary rational conformal field theory (or D-branes in string theory): in particular, finding NIM-reps is equivalent to solving Cardy's equation in these cases. Some recent papers have continued this work \cites{CRSS,YZL}, showing that this interest keeps up after time.

Given a rational, $c_2$-cofinite vertex operator algebra describing the chiral symmetries of a rational conformal field theory, its category of representations is a modular tensor category \cite{Hua,HLZ}. In this paper we go for a slightly more general setting than modular and focus on the study and classification of NIM-reps for two different families of fusion categories, pointed and near-group. We also include the case of the modular tensor category $A \left( 1,k \right)_{\frac{1}{2}}$, thus covering a certain amount of modular tensor categories known to date (see e.g. \cites{BGNPRW,BNRW,BOM,BPRZ,Br,NRWW,RSW}, more recently \cite{NRW}, and references thereof). Our main motivation for undertaking this work is to test the potential of NIM-reps as a detection tool for families of algebra objects in fusion categories that represent their corresponding module categories. In this way, we generalise the results of \cite{DB} for any pointed fusion category (recovering the expected group algebras as described by \cites{Nat,Ost}), extend our understanding of algebra objects in near-group categories as discussed in \cites{MM,Gal}, and recover the results outlined by \cite{EK} for $\mathbb{Z}_+$-modules and by \cite{Ost2} for algebra objects in $A \left( 1,k \right)_{\frac{1}{2}}$. Furthermore, there is an interesting conjecture relating the exponents of NIM-reps and those of the modular invariants of the category they live on. In an appendix, we test explicitly this conjecture for unitary modular tensor categories up to rank 4 (which by \cite{RSW} include examples of the three families of categories considered in this paper for which we compute their respective NIM-reps).

The structure of the paper is as follows. In Section \ref{prelim} we introduce all the necessary background on fusion and module categories as well as algebra objects, and also on $\mathbb{Z}_+$-modules and NIM-reps. In Section \ref{NIMrepscalcs} we compute the NIM-reps for each category, and also the algebra objects derived from those. Appendix \ref{modularinvs} includes the calculations testing the exponents conjecture for modular tensor categories of rank up to 4.

\subsection*{Acknowledgements}
SH is supported by the Engineering and Physical Sciences Research Council. ARC is supported by Cardiff University. Devi Young's contribution was supported by the Cardiff University On-Campus Internship Scheme 21/22.

\section{Preliminaries}\label{prelim}

Throughout this paper, we fix $\mathbbm{k}$ to be an algebraically closed field. In this section, we collect some basic definitions necessary for our work, from group theory (see e.g. \cite{Cam}), fusion and modular tensor categories, $\mathbb{Z}_+$-rings and module categories {\`a} la \cite{EGNO}, to NIM-reps \cites{Behrend,BPRZ,GannonNIM,Gannon,Gannonbook} and their relation to the previous.

\subsection{Group Actions}

\begin{definition}
Let $G$ be a group and $S$ a set.
A \textit{G-action on $S$} is a binary operation $\ast \colon G \times S \to S$ such that,  for all $s\in S$,
\begin{equation*}
    e \ast s =s \text{, and }  (g \cdot h)\ast s = g \ast (h \ast s),
\end{equation*}
where $e\in G$ is the group identity element. A set with such an action is called a \textit{G-set}.
\end{definition}

\begin{definition}
Let $S$ be a $G$-set, and $s \in S$.
\begin{itemize}
    \item[-] The \textit{orbit of G through s} is the subset of $S$ defined by
\begin{equation*}
    \mathrm{Orb}(s) = \{g\ast s \vert g\in G\}.
\end{equation*}
\item [-] The $G$-action on $S$ is called \textit{transitive} if $\mathrm{Orb}(s) = S$.
    \item[-] The \textit{stabiliser of s} is the subgroup of $G$ defined by 
\begin{equation*}
    \mathrm{Stab}(s)=\{g\in G \vert g\ast s = s\}.
\end{equation*}

\end{itemize}
\end{definition}
For two elements $s,t \in S$, the orbits $\mathrm{Orb}(s),\mathrm{Orb}(t)$ are either equal or disjoint. Hence the set $S$ can be partitioned into a collection of transitive $G$-sets.

\begin{proposition}\label{finorb}
Let $S$ be a $G$-set, and take $s\in S$. There is an isomorphism of $G$-sets between the orbit $\text{Orb}(s)$ and the set of left cosets $G/\text{Stab}(s)$, where $G$ acts on the set of left cosets by left-translation.
\end{proposition}
\begin{proposition} \label{prop-conjgroup}
    Two left coset $G$-sets $G/H$,$G/K$ are isomorphic as $G$-sets if and only if $H, K$ are conjugate as subgroups of $G$.
\end{proposition}
As a result of these propositions, we can study all $G$-sets by studying the $G$-sets of left cosets for all conjugacy classes of subgroups in $G$.

\subsection{Fusion Categories}

A \textit{monoidal category} $\cC$ consists of a tuple $\left( \cC,\otimes, \mathbbm{1},\alpha,\lambda,\rho \right)$ where $\cC$ is a category, $\otimes \colon \cC \times \cC \to \cC$ is a bifunctor, $\mathbbm{1} \in \Ob \left( \cC \right)$, $\alpha_{X,Y,Z} \colon \left( X \otimes Y \right) \otimes Z \to X \otimes \left( Y \otimes Z \right)$ is a natural isomorphism for each $X, Y, Z \in \Ob \left( \cC \right)$, and $\lambda_X \colon \mathbbm{1} \otimes X \to X$ and $\rho_X \colon X \otimes \mathbbm{1} \to X$ are natural isomorphisms for all $X \in \Ob \left( \cC \right)$,
satisfying coherence axioms (pentagon and triangle).

A monoidal category is called \textit{rigid} if it comes equipped with left and right dual objects --- that means, for every $X \in \Ob \left( \cC \right)$ there exists respectively an object $X^* \in \Ob \left( \cC \right)$ with evaluation and coevaluation maps $\ev_X \colon X^* \otimes X \to \mathbbm{1}$ and $\coev_X \colon \mathbbm{1} \to X \otimes X^*$, as well as an object ${}^*X \in \Ob \left( \cC \right)$ with evaluation and coevaluation maps $\evr_X \colon X \otimes {}^*X \to \mathbbm{1}$ and $\coevr_X \colon \mathbbm{1} \to {}^*X \otimes X$ satisfying in both cases the usual conditions.

A $\mathbbm{k}$-linear abelian category $\cC$ is \textit{locally finite} if, for any two objects $V,W \in \Ob \left( \cC \right)$, $\Hom_\cC \left( V,W \right)$ is a finite-dimensional $\mathbbm{k}$-vector space and every object has a finite filtration by simple objects. Further, we say $\cC$ is \textit{finite} if there are finitely many isomorphism classes of simple objects. A \textit{tensor category} is a locally finite, rigid, monoidal category such the the tensor product is $\mathbbm{k}$-linear in each slot and the monoidal unit is a simple object of the category.

At this point, it is useful to introduce the following notion. Given an abelian category $\cC$, the \textit{Grothendieck group} $\mathrm{Gr} \left( \cC \right)$ of $\cC$ is the free abelian group generated by isomorphism classes $X_i$ of simple objects in $\cC$. If $X$ and $Y$ are objects in $\cC$ such that $Y$ is simple then we denote as $\left[ X \colon Y\right]$ the multiplicity of $Y$ in a Jordan-H{\"o}lder series of $X$. To any object $X$ in $\cC$ we can canonically associate its class $\left[ X \right] \in \mathrm{Gr} \left( \cC \right)$ given by the formula:
\begin{equation*}
    \left[ X \right]=\sum_i \left[ X \colon X_i \right] X_i.
\end{equation*}

A monoidal category $\cC$ is called \textit{braided} if it comes equipped with natural isomorphisms $c_{X,Y} \colon X \otimes Y \to Y \otimes X$, for all $X, Y \in \Ob \left( \cC \right)$, called the \textit{braiding}, that are compatible with the monoidal structure of the category. This means, the braiding satisfies the so-called \textit{hexagon identities} for any three objects $X, Y, Z \in \Ob \left( \cC \right)$:
\begin{equation*}
    \xymatrix{& X \otimes \left( Y \otimes Z \right) \ar[rr]^{c_{X, Y \otimes Z}} && \left( Y \otimes Z \right) X \ar[dr]^{\alpha_{Y,Z,X}} & \\
    \left( X \otimes Y \right) \otimes Z \ar[ur]^{\alpha_{X,Y,Z}} \ar[dr]_{c_{X,Y} \otimes \ide_Z} &&&& Y \otimes \left( Z \otimes X \right) \\
    & \left( Y \otimes X \right) \otimes Z \ar[rr]^{\alpha_{Y,X,Z}} && Y \otimes \left( X \otimes Z \right) \ar[ur]_{\ide_Y \otimes c_{X,Z}}
    }
    \end{equation*}
    \begin{equation*}
    \xymatrix{ & \left( X \otimes Y \right) \otimes Z \ar[rr]^{c_{X \otimes Y,Z}} && Z \otimes \left( X \otimes Y \right) \ar[dr]^{\alpha^{-1}_{Z,X,Y}} & \\
    X \otimes \left( Y \otimes Z \right) \ar[ur]^{\alpha^{-1}_{X,Y,Z}} \ar[dr]_{\ide_X \otimes c_{Y,Z}} &&&& \left( Z \otimes X \right) \otimes Y \\
    & X \otimes \left( Z \otimes Y \right) \ar[rr]^{\alpha^{-1}_{X,Z,Y}} && \left( X \otimes Z \right) \otimes Y \ar[ur]_{c_{X,Z} \otimes \ide_Y}
    }
\end{equation*}

A \emph{ribbon category} is a braided tensor category $\cC$ together with a \emph{ribbon twist}, i.e., a natural isomorphism $\theta_X\colon X\to X$ which satisfies 
\begin{align*}
\theta_{X \otimes Y} = (\theta_X \otimes \theta_Y)   c_{Y,X}  c_{X,Y},\qquad  \theta_\one = \ide_\one,\qquad (\theta_X)^* = \theta_{X^*}.
\end{align*}

In order to define modular tensor categories, we require the notion of non-degeneracy of a braided category. We say that an object $X$ \emph{centralizes} another object $Y$ of $\cC$ if 
$$c_{Y,X}c_{X,Y}=\ide_{X\otimes Y}.$$
A braided finite tensor category $\cC$ is \emph{non-degenerate}  if the only objects $X$ that centralize \emph{all} objects of $\cC$ are of the form $X=\one^{\oplus n}$ \cite{EGNO}*{Section~8.20}. 
Equivalently, $\cC$ is non-degenerate if and only if it is \emph{factorizable}, i.e., there is an equivalence of braided monoidal categories $\cZ(\cC)\simeq \cC^{\rev}\boxtimes\cC$, where $\cC^\rev$ is $\cC$ as a tensor category, but with reversed braiding given by the inverse braiding \cite{Shi1}. If $\cC$ is a \emph{fusion category} (i.e., a semi-simple finite tensor category) then the above notion of non-degeneracy is equivalent to the commonly used condition that the $S$-matrix is non-singular.

\begin{definition}[{\cites{KL,Shi1}}]
A braided finite tensor category is \emph{modular} if it is a non-degenerate ribbon category.
\end{definition}

\subsection{$\mZ_+$-Rings}

Denote as $\mZ_{+}$ the semi-ring of positive integers with zero.
\begin{definition}
    Let $R$ be a ring which is free as a $\mZ$-module.
    \begin{enumerate}
        \item A $\mZ_+$-\textit{basis} of $R$ is a basis $B=\lbrace b_i \rbrace_{i \in I}$ such that $b_i b_j=\sum\limits_{k \in I} c_{ij}^k b_k$, where $c_{ij}^k \in \mZ_+$.
        \item A $\mZ_+$-\textit{ring} is a ring with a fixed $\mZ_+$-basis and with an identity $1$ which is a non-negative linear combination of the basis elements. If $1$ is a basis element, then it is called a \textit{unital} $\mathbb{Z}_+$-\textit{ring}.

        \item Given a $\mathbb{Z}_+$-ring $\left( R,B \right)$, a $\mathbb{Z}_+$-\textit{module} is an $R$-module $A$ with a fixed $\mathbb{Z}$-basis $M$ such that for any $m_i \in M$, $b_j \in B$, then $b_i \cdot m_j=\sum\limits_k a_{i,l}^k m_k$ with $a_{i,l}^k \in \mathbb{Z}_+$.
    \end{enumerate}
\end{definition}

\begin{example} 
For $\cC$ a fusion category with $X_i$ the representatives of the isomorphism classes of simple objects, the tensor product on $\cC$ induces a natural multiplication on $\mathrm{Gr} \left( \cC \right)$ defined by the formula:
\begin{equation*}
    X_i X_j:=\left[ X_i \otimes X_j \right]=\sum\limits_{k \in I} \left[ X_i \otimes X_j \colon X_k \right] X_k
\end{equation*}
where $i,j \in I$. This multiplication is associative, and thus $\mathrm{Gr} \left( \cC \right)$ is a $\mathbb{Z}_+$-ring with unit $\left[ \one \right]$. $\mathrm{Gr} \left( \cC \right)$ is called the \textit{Grothendieck ring} of $\cC$.
\end{example}

Let $(R,B)$ be a $\mathbb{Z}_+$-ring, and let $i \in I_0 \subset I$ such that $b_i$ appears in the decomposition of 1. Then, let $\tau \colon R \to \mathbb{Z}$ denote the group homomorphism defined by:
\begin{equation*}
    \tau \left( b_i \right)= \begin{cases}
    1 & i \in I_0, \\
    0 & i \notin I_0.
    \end{cases}
\end{equation*}

\begin{definition}
    A $\mathbb{Z}_+$-ring $(R,B)$ is called a \textit{based ring} if there exists an involution $({})^* \colon I \to I$, $i \mapsto i^*$ of the label set $I$ such that the induced map
    \begin{equation*}
        a=\sum\limits_{i \in I} a_i b_i \mapsto a^*=\sum_{i \in I} a_i b_{i^*},
    \end{equation*}
    where $a_i \in \mathbb{Z}$, is an anti-involution of the ring $R$ and
    \begin{equation*}
    \tau \left( b_i b_j \right)= \begin{cases}
    1 & i=j^*, \\
    0 & i \neq j^*.
    \end{cases}
\end{equation*}
\end{definition}

We arrive to one of the main definitions of this paper:

\begin{definition}
    A \textit{fusion ring} is a unital, based ring of finite rank.
\end{definition}

In order to later introduce the NIM-reps, it is convenient to observe the following property:

\begin{proposition}\label{rigprop}(Rigidity property)
   A fusion ring $(R,B)$ can be equipped with a symmetric bilinear form $(-,-)$ satisfying $(b_ib_j,b_k) = (b_j,b_{i^*}\cdot b_k)$.
\end{proposition}
\begin{proof}
    Let $(-,-)$ be the symmetric bilinear form defined by the condition $( b_i,b_j )=\delta_{i,j}$. Then the property of $\tau(b_ib_j) = \delta_{i,j^*}$ can be reformulated as 
    \begin{equation*}
        (1,b_ib_j) = (b_i,b_{j^*}) = (b_{i^*},b_j)
    \end{equation*}
    It is clear that $(1,(b_ib_j)b_k)\tau((b_ib_j)b_k) = c_{ij}^{k^*}$, and by the cyclic invariance we get that 
    \begin{equation*}
        (1,(b_ib_j)b_k)= (1,b_i(b_jb_k))
    \end{equation*}
    Thus, we find that
    \begin{equation*}
        ((b_ib_j)^*,b_k) = (b_{j^*}b_{i^*},b_k) = (b_{i^*},b_jb_k)
    \end{equation*}
    where the first equality uses that the induced map is an anti-involution. Relabelling of the indices gives the stated result.
\end{proof}

\begin{notation}
Whenever clear from the context, we may refer to a fusion ring simply as $R$ instead of $\left( R,B \right)$.
\end{notation}

\begin{example}\label{Grothring}
    Let $\cC$ be a semi-simple rigid monoidal category. Then the Grothendieck ring is a fusion ring, with the simple objects acting as the basis. The involution is then taking the dual of an object and the symmetric bilinear form is $(X,Y)=\mathrm{dim}_{\mathbbm{k}} \left( \mathrm{Hom}_\cC \left( X,Y \right)\right)$, for simple objects $X,Y \in \text{Ob(C)}.$
\end{example}

The following explicit examples will be studied in detail in later Sections.

\begin{example}[\emph{Group rings}]\label{grouprng}
    Take a finite group G, and construct the group ring $\mZ G$ where addition is linear and multiplication is given by the group operation. Then $(\mZ G, G)$ is a fusion ring with involution $g^*= g^{-1}$. 
\end{example}

\begin{example}[\emph{Ising fusion ring}]\label{ising}
    Let $B=\{1, X, Y\}$, and $R$ the integer span $\mZ B$ with addition defined linearly and multiplication given by the fusion rules
    \begin{equation*}
        X^2 = 1 + Y, \quad Y^2 = 1 \quad XY = YX = X.
    \end{equation*}
    $(\mZ B,B)$ is a fusion ring with the self-dual involution $X^* = X , Y^* = Y.$
    
\end{example}

\subsection{NIM-Reps}

In this subsection we introduce the main character of this article. 

\begin{definition}
Let $(R,B)$ be a fusion ring. A \textit{non-negative integer matrix representation} (NIM-rep for short) of $(R,B)$ is a $\mZ_+$-module $(A,M)$ that satisfies the following condition;
\begin{itemize}
    \item[-] \emph{(Rigidity condition):} let A have a symmetric bilinear form $(-,-)$ defined by 
\begin{equation*}
(m_l,m_k)=\delta_{l,k}
\end{equation*}
for any $l,k \in L$.
Then we must have, for any $i\in I, l,k\in L$ 
\begin{equation*}
(b_i\act m_l,m_k) = (m_l,b_{i^*}\act m_k).
\end{equation*}
\end{itemize}
\end{definition}

\begin{remark}
    In this definition, unlike \Cref{rigprop}, rigidity is a condition, not a property.
\end{remark}

\begin{example}
    A fusion ring can always be considered as a NIM-rep of itself, with the module action simply the ring multiplication.
\end{example}

\begin{example}\label{catnim}
(Follow-up from Definition \ref{Grothring}) If we take a semi-simple module category $\cM$ over $\cC$, then $\mathrm{Gr}(\cM)$ is a NIM-representation of $\mathrm{Gr}( \cC )$.
\end{example}

\begin{remark}\label{rigidity}
    Note that, for $a \in A, m_l \in M$, the symmetric bilinear form $(a,m_l)$ counts the mulitplicity of $m_l$ in the basis decomposition of $a$. We then immediately see that $(a,a)=1$ if and only if $a \in M$.
\end{remark}

Viewed as $\mathbb{Z}_+$-modules, it is straightforward to define the direct sum of NIM-reps: given two NIM-reps $(A,M),(A',M')$ over a fusion ring $(R,B)$ the \textit{direct sum of NIM-reps} is the $R$-module $A \oplus A'$ with a distinguished basis $M \oplus M'$. Other basic notions like \textit{sub-NIM-rep} are defined in a similar way.

\begin{definition}[\cite{EGNO} Section 3.4, \cite{Ost2} Lemma 2.1.]
   A NIM-rep is called \textit{irreducible} if it has no proper sub-NIM-reps\footnote{One can also have the notion of \textit{indecomposable NIM-rep}, meaning one which is not isomorphic to a non-trivial direct sum of NIM-reps. Since we are working over fusion rings, these two notions are equivalent.}.
\end{definition}

\begin{remark}\label{nim0}
Suppose we have a NIM-rep $(A,M)$ over a fusion ring $(R,B)$ that satisfies $b_i \act m_j = 0_R$ for some $b_i \in B, m_j \in M$. The rigidity condition then imposes that $m_j =0_R \in M$. However, the only way $0_R$ appears in the NIM-rep basis is if $\{0_R\} = M$, (i.e this NIM-rep is the trivial NIM-rep). We shall remove this NIM-rep from future considerations. 
\end{remark}

\begin{definition}[\cite{DB}] Let $(A,M),(A',M')$ be two NIM-reps over a fusion ring $(R,B)$. A \textit{NIM-rep morphism} is a function $\psi \colon M \to M'$ inducing a $\mZ$-linear map between the modules. If $\psi$ is a bijection, and the induced map is an isomorphism of $R$-modules, then we say that the NIM-reps are \textit{equivalent}.
\end{definition}

\begin{notation}
Similarly to the case of a fusion ring, we may refer to a NIM-rep simply as $M$ instead of $\left( A,M \right)$.
\end{notation}

We can visually express the data of a NIM-rep in the following way.
For a given NIM-rep $(A,M)$ over a fusion ring $(R,B)$, a \textit{NIM-graph} (sometimes also called in the literature \textit{`fusion graphs'}) is constructed with a node for each element of the basis $M$, and a directed arrow with source $m_l$ and target $m_k$, labelled by an element $b_i \in B$, for every copy of $m_k$ in $b_i \vartriangleright m_l$.  Every node in a NIM-graph will have a self-loop labelled by the ring identity, which we omit for simplicity.

\begin{example}
    If we consider the Ising fusion ring from \Cref{ising} as a NIM-rep over itself, the corresponding NIM-graph is given by;
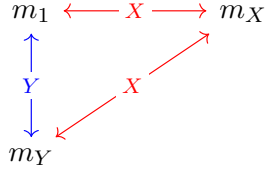
\begin{center}
    \begin{tikzcd}
        m_1 \arrow[rr,red, leftrightarrow , "X" description] \arrow[dd, blue, leftrightarrow, "Y" description] & & m_X \arrow[ddll, red, leftrightarrow, "X" description] \\
        & & \\
        m_Y   &  &  
\end{tikzcd}
\end{center}
\end{example}

\begin{remark}\label{rmkNIMreps}
\begin{itemize}
    \item[-] The NIM-graph allows us to visualise irreducibility of the corresponding NIM-rep, as a NIM-rep is irreducible if and only if the NIM-graph is connected.
    \item[-] In \cites{Gannon,Gannonbook}, NIM-reps are defined equivalently as an assignment $a \mapsto N_a$ to each $a \in \Phi=\mathrm{Ob} \left( \cC \right)$ of a matrix $N_a$ with non-negative integer entries, satisfying several compatibility conditions. We will not use this description in the main text but it will be useful in Appendix \ref{modularinvs}.
\end{itemize}
\end{remark}

\subsection{Module Categories, Algebra Objects and NIM-Reps}

In this section, unless specified $\cC=\left( \cC,\otimes, \mathbbm{1},\alpha,\lambda,\rho \right)$ is a fusion category. 
\begin{definition}
An \textit{algebra} in $\cC$ is a triple $\left( A,m,u \right)$, with $A \in \Ob \left( \cC \right)$, and  $m \colon A \otimes A \to A$ (multiplication), $u \colon \mathbbm{1} \to A$ (unit) being morphisms in $\cC$,
    satisfying unitality and associativity constraints: 
    \begin{align*} 
      \hspace{.5in}  m (m \otimes \ide_A) = m(\ide_A \otimes m) \alpha_{A,A,A}, \quad \quad
       m (u \otimes \ide_A) = \lambda_A, \quad   \quad m(\ide_A \otimes u) = \rho_A. 
    \end{align*}
\end{definition}

\begin{example}
    Let $\cC$ be a tensor category, then $\one$ is an algebra. In fact, for any $X \in \mathrm{Ob}(\cC)$ the object $A=X \otimes X^*$ has a natural structure of an algebra with unit $u=\mathrm{coev}_X$ and multiplication $m=\ide_X \otimes \mathrm{ev}_X \otimes \ide_{X^*}$.
\end{example}

\begin{definition}
\begin{enumerate}
     \item[(a)] An algebra $A$ in $\cC$ is \textit{indecomposable} if it is not isomorphic to a direct sum of non-trivial algebras in $\cC$.
        \item[(b)] An algebra $A$ in $\cC$ is \textit{separable} if there exists a morphism $\Delta' \colon A \to A \otimes A$ in $\cC$ so that  $m \Delta' = \ide_A$ as maps in $\cC$ with
    \begin{equation*}
               \left( \ide_A\otimes m \right) \alpha_{A, A, A} \left( \Delta'\otimes \ide_A \right) = \Delta' m = \left( m \otimes \ide_A \right) \alpha^{-1}_{A, A, A} \left( \ide_A\otimes \Delta' \right).
   \end{equation*}
     \end{enumerate}
\end{definition}

Just like in abstract algebra, one can construct the related notion of a module over an algebra in the following way.

\begin{definition} Take $A:=(A,m_A,u_A)$, an algebra in $\cC$. 
A {\it right $A$-module in $\cC$} is a pair $(M, \rho_M)$, where $M \in \mathrm{Ob}(\cC)$, and $\rho_M:=\rho_M^A \colon M \otimes A \to M$ is a morphism in $\cC$  so that $$\rho_M(\rho_M \otimes \ide_A) = \rho_M(\ide_M \otimes m_A)\alpha_{M,A,A} \quad \text{ and } \quad r_M = \rho_M(\ide_M \otimes u_A).$$ A {\it morphism} of right $A$-modules in $\cC$ is a morphism $f \colon M \to N$ in $\cC$ so that $f\rho_M  = \rho_N (f \otimes \ide_A)$. Right $A$-modules in $\cC$ and their morphisms form a category, which we denote by~$\Mod_\cC - A$. 
The categories $A-\Mod_\cC$ of {\it left $A$-modules $(M, \lambda_M:=\lambda_M^A \colon A \otimes M \to M)$  and $A-\mathrm{Bimod}_\cC$ of {\it $A$-bimodules} in $\cC$} are defined likewise.
\end{definition}

We want to relate these categories of modules to the following notion:
\begin{definition}
    Let $\cC$ be a monoidal category. A {\it left module category} over $\cC$ is a category $\cM$ equipped with an {\it action} (or {\it module product}) bifunctor $\otimes \colon \cC \times \cM \to \cM$ and natural isomorphisms 
    \begin{equation*}
            m_{X,Y,M} \colon \left( X \otimes Y \right) \otimes M \to X \otimes \left( Y \otimes M \right),
    \end{equation*}
    called the {\it module associativity constraint} and
    \begin{equation*}
            l_{M} \colon \one \otimes M \to M, \quad \quad r_{M} \colon M \otimes \one \to M
    \end{equation*}
    called the {\it unit constraints}, such that both {\it pentagon diagram}: 
    \begin{equation*}
    \xymatrix{
    & \left( \left( X \otimes Y \right) \otimes Z \right) \otimes M \ar[dl]_{a_{X,Y,Z} \otimes \ide_M} \ar[dr]^{m_{X \otimes Y,Z,M}} & \\
    \left( X \otimes \left( Y \otimes Z \right) \right) \otimes M \ar[d]^{m_{X,Y \otimes Z,M}} && \left( X \otimes Y \right) \otimes \left( Z \otimes M \right) \ar[d]_{m_{X,Y,Z \otimes M}} \\
    X \otimes \left( \left( Y \otimes Z \right) \otimes M \right) \ar[rr]^{\ide_X \otimes m_{Y,Z,M}} && X \otimes \left( Y \otimes \left( Z \otimes M \right) \right)
    }
    \end{equation*}
    and the {\it triangle diagram}:
    \begin{equation*}
        \xymatrix{
    \left( X \otimes \one \right) \otimes M \ar[rr]^{m_{X,\one,M}} \ar[dr]_{r_X \otimes \ide_M} && X \otimes \left( \one \otimes M \right) \ar[dl]^{\ide_X \otimes l_M} \\
    & X \otimes M &
        }
    \end{equation*}
    are commutative $\forall X,Y,Z \in \text{Ob}(\cC)$ and $M \in \cM$.
\end{definition}
In a similar way, one defines a {\it right} $\cC$-module category.

A tensor category is the simplest example: it is a module category over itself. A convenient, less trivial example for us is the following:

    \begin{proposition}[Proposition 7.8.10, \cite{EGNO}]
        $\Mod_\cC -A$ is a left $\cC$-module category.
    \end{proposition}

    In fact, given certain conditions one can go one step further.

    \begin{lemma}[Proposition 7.8.30, \cite{EGNO}]
        Let $A$ be a separable algebra in a fusion category $\cC$. Then the category $\Mod_\cC-A$ of right $A$-modules in $\cC$ is also semisimple.
    \end{lemma}

We now describe how to extract algebra objects from certain NIM-reps. We state the following theorem using results from \cite{EGNO}*{Section 7.10}:
\begin{theorem}\label{thm-EGNOmodule}
    Let $\cC$ be a fusion category, $\cM$ an indecomposable, semisimple $\cC$-module category, and $N \in \mathrm{Ob} \left(\cM \right)$ such that $[N]$ generates $\mathrm{Gr}(\cM)$ as a based $\mZ_+$-module over $\mathrm{Gr}(\cC)$. Then there is an equivalence $\cM \simeq \Mod_\cC-A$ of $\cC$-module categories, where $A= \underline{\mathrm{Hom}}(N,N)$.
\end{theorem}
\begin{definition}
    We shall call a NIM-rep that comes from $\Mod_\cC-A$ in the sense of \Cref{catnim} as \textit{admissible}.
\end{definition}

Suppose we are in the setup required for \Cref{thm-EGNOmodule}.
\begin{proposition}\label{prop-NIMadmis}
    A NIM-rep $(A,M)$ over the fusion ring $(R,B)$ constructed from some fusion category $\cC$ is admissible only if there exists an element $m_0$ in the basis $M$ such that for every element $m_i \in M$, there exists an element $b_j \in B$ such that $b_j \act m_0 = m_i$.
\end{proposition}
\begin{proof}
    This follows straightforwardly from the conditions in \Cref{thm-EGNOmodule}. Assuming $(A,M)$ is admissible from $\cM$, we identify $m_0$ with the class $[N]$ that generates Gr$(\cM)$ as a $\mZ_+$-module. Then for any $m_i \in M$ we must have some element $X \in R$ such that $X \act m_0 = m_i$. But if we write $X$ in terms of basis elements, it is clear from the non-negativity of the NIM-rep that there must be a distinguished $b_i \in B$ such that $b_i \act m_0 = m_i$. We then also get that $m_0 \in M$ from the rigidity condition.
\end{proof}
\begin{proposition}
    Let $M$ be an admissible NIM-rep over the fusion ring $(R,B)$ constructed from $\cC$. Then the decomposition of the algebra $A=\underline{\mathrm{Hom}}(N,N)$ is given by $\bigoplus_{i \in I}a_i b_i$, where $a_i$ is the number of self-loops of $m_0$ labelled by $b_i$ in the NIM-graph of $(A,M)$. 
\end{proposition}
\begin{proof}
    Using the isomorphism from \cite{EGNO}*{Equation 7.21} applied to the algebra $A$, we have that 
    \begin{equation*}
        \Hom_{\cC}(X, A) \cong \Hom_{\cM}(X \act N,N).
    \end{equation*}
    By Schur's Lemma, if we restrict $X$ to the simple objects of $C$, then $X$ appears in the decomposition of $A$ if and only if $N$ is in the decomposition of $X \act N$. But by restricting to the NIM-reps picture, and the identification of $m_0$ with $N$, we see that this occurs exactly when $X$ labels a self-loop on $m_0$. This gives the result.
\end{proof}

\section{NIM-Representations}\label{NIMrepscalcs}

In the following subsections we compute explicitly the NIM-reps of fusion rings associated to relevant examples of families of modular and fusion categories. We also  extract from these any algebra objects and compare to existing results.

\subsection{Group Rings}\label{sec:groupNIM}
Let $G$ be a finite group. In this section, we will focus on classifying all possible NIM-reps over the group fusion rings $R(G):=(\mZ G,G)$ described in \Cref{grouprng}. 
 
\begin{proposition}
Let $M$ be a NIM-Rep over the fusion ring $R(G)$. The NIM-rep module action restricts to a group action on $M$.
\end{proposition}

\begin{proof}
   The NIM-rep module action will restrict to a group action on $M$ if every element $g\vartriangleright m_l$ is in the basis of $M$. This can be seen as 
   \begin{equation*}
       (g \vartriangleright m_l, g \vartriangleright m_l) = (m_l,g^{-1}\vartriangleright(g \vartriangleright m_l)) = (m_l,m_l) = 1
   \end{equation*}
   and so $g \vartriangleright m_l$ is in the basis of $M$ by \Cref{rigidity}.
\end{proof}

If we restrict ourselves to irreducible NIM-reps, we get the following result. 

\begin{proposition}\label{groupclassify}
    Irreducible NIM-reps of the group fusion ring $R(G)$ correspond to transitive group actions of G. 
\end{proposition}
\begin{proof}
    If a NIM-rep $(A,M)$ over $R(G)$ is not irreducible, then its corresponding group action will always be partitioned into $G$-orbits by restricting the action to the NIM-reps that sum to give $(A,M)$, Thus the group action is transitive only if the NIM-rep is irreducible.

    Conversely, if the group action is not transitive, then we can write it as the sum of some finite combination of G-actions $\act_i:G \times M_i \to M_i$. It is easy to see that each $M_i$ is a $\mZ$-basis for another NIM-rep over $R(G)$. Hence the NIM-rep is irreducible only if the group action is transitive.
\end{proof}

We can thus explicitly describe the structure of such NIM-reps of $R(G)$. By \Cref{finorb}, the basis elements of an irreducible NIM-rep $(A,M)$ are parametrised by the left cosets of $H$ in $G$, for some subgroup $H \subseteq G$ i.e If we let $\{g_i\}_{1\leq i \leq |G:H|}$ be a set of coset representatives, then $M = \{m_{g_i}\}_{1\leq i\leq |G:H|}$.

The NIM-rep action is then given by the induced $G$-action on this set of left cosets, $$X \act m_{g_i} = m_{g_j}$$ where $Xg_i \in g_j H$. We shall write $M(H)$ for such a NIM-rep.

\begin{proposition}
   Two NIM-reps $M(H),M(K)$ over $R(G)$ are isomorphic if and only if $H,K$ are conjugate subgroups of $G$. 
\end{proposition}
\begin{proof}
  This follows immediately by combining \Cref{prop-conjgroup} and \Cref{groupclassify}.
\end{proof}

\begin{example}\label{ex-klein4} 
(NIM-reps of $R(\mZ_2 \times \mZ_2)$)

The Klein-four group has presentation
\begin{equation*}
    \mZ_2 \times \mZ_2 = \{ a,b,c | a^2 = b^2 = c^2 = abc =e\}
\end{equation*}

There are 3 isomorphism classes of subgroups in $\mZ_2 \times \mZ_2$;
\begin{itemize}
    \item $\mZ_2 \times \mZ_2$ as a subgroup of itself; Then $M(\mZ_2 \times \mZ_2)$ has a single basis element corresponding to the single coset representative $e$. The NIM-rep graph is given by 
    
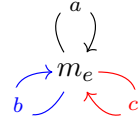
\begin{center}
        \begin{tikzcd}[nodes={inner sep=0.5pt,minimum size=1ex,circle}]
            m_e \arrow[out=120, in = 60, loop, "a" description] \arrow[out=240, in=180, loop, "b" description,blue] \arrow[out=0, in=300,loop,"c" description,red]
        \end{tikzcd}
    \end{center}
    \item Isomorphism class of $\mZ_2$; There are 3 conjugacy classes of subgroups in this case; $H_1 = \{e,a\}, H_2=\{e,b\}, H_3 = \{e,c\}$. The NIM-reps $M(H_1),M(H_2),M(H_3)$ have two basis elements parameterised by coset representatives $\{e,b\}, \{e,c\}, \{e,a\}$ respectively.
    
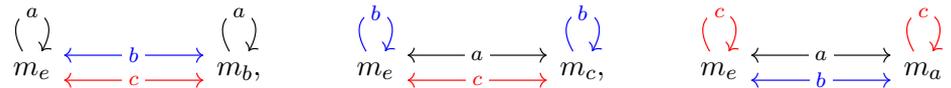
\begin{center}
        \begin{tikzcd}
            m_e \arrow[out=120,in=60,loop,"a" description] \arrow[rr,leftrightarrow, "b" description,blue, yshift = 1ex] \arrow[rr,leftrightarrow,"c" description,red, yshift = -1ex] & & m_b \arrow[out=120,in=60,loop, "a" description]  ,& m_e \arrow[out=120,in=60,loop,"b" description,blue] \arrow[rr,leftrightarrow, "a" description, yshift = 1ex] \arrow[rr,leftrightarrow,"c" description,red, yshift = -1ex] & & m_c \arrow[out=120,in=60,loop, "b" description,blue] ,&m_e \arrow[out=120,in=60,loop, "c" description,red] \arrow[rr,leftrightarrow, "b" description,blue, yshift = -1ex] \arrow[rr,leftrightarrow,"a" description, yshift = 1ex] & & m_a \arrow[out=120,in=60,loop, "c" description,red]
        \end{tikzcd}
    \end{center}
        
    \item The trivial subgroup $H=\{e\}$; The basis elements of $M(H)$ are simply parameterised by elements of $\mZ_2 \times \mZ_2$.
    
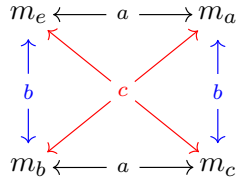
\begin{center}\vspace{-15pt}
        \begin{tikzcd}[nodes={inner sep=1pt,minimum size=0.5ex,circle}]
            m_e \arrow[rr, leftrightarrow, "a" description] \arrow[dd, leftrightarrow, "b" description, blue] \arrow[ddrr, leftrightarrow, red] & & m_a \arrow[dd, leftrightarrow, "b" description, blue] \arrow[ddll, leftrightarrow, "c" description, red] \\
             & & \\
            m_b \arrow[rr, leftrightarrow, "a" description] & & m_c
        \end{tikzcd}
    \end{center}
\end{itemize}
\end{example}
\vspace{-30pt}

\begin{example}\label{ex-d3} 
(NIM-reps of $R(D_3)$)

     Consider the dihedral group $D_3$, with presentation
     \begin{equation*}
         D_3 = \{x,a \; | \; x^2=a^3=e, \; xa=a^{-1}x\}
     \end{equation*}
There are four conjugacy classes of subgroups of $D_3$, giving four isomorphism classes of NIM-reps;
\begin{itemize}
\item $D_3$ as a subgroup of itself. This NIM-graph simple consist of a single basis element, with each group ring element acting trivially.
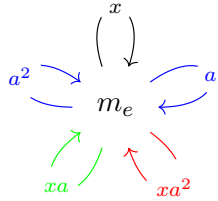
\begin{center}\vspace{-5pt}
        \begin{tikzcd}[nodes={inner sep=6pt,minimum size=0.1ex,circle}]
            m_e \arrow[out=108, in =72, loop, "x" description] \arrow[out=36, in=0, loop, "a" description,blue] 
            \arrow[out=324, in=288,loop,"xa^2" description,red]
            \arrow[out=252, in=216,loop, "xa" description, green]
            \arrow[out=180, in=144,loop, "a^2" description, blue]
        \end{tikzcd}
    \end{center}
    \item The isomorphism class of $\mZ_3$, given by $H =\{e,a,a^2\}$.
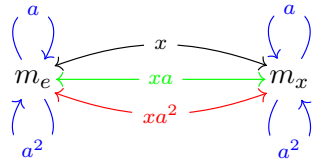
    \begin{center}\vspace{-20pt}
        \begin{tikzcd}[nodes={inner sep=0.5pt,minimum size=1ex,circle}]
            m_e \arrow[out=120,in=60,loop, "a" description,blue, distance=1cm] \arrow[in=240,out=300,loop, "a^2" description ,blue, distance=1cm]\arrow[rrr,leftrightarrow,bend left=20, "x" description, shift left=0.8] \arrow[rrr,leftrightarrow, "xa" description, green] \arrow[rrr,leftrightarrow, "xa^2" description, red, bend right=20, shift right=0.8] & & & m_x \arrow[out=60,in=120,loop, "a" description,blue,distance=1cm] \arrow[out=240,in=300,loop, "a^2" description,blue,distance=1cm] 
        \end{tikzcd}
    \end{center}\vspace{-15pt}
    \item The isomorphism class of $\mZ_2$ is given by 3 conjugate subgroups, $H =\{e,x\}, \{e,xa\},\{e,xa^2\}$. This gives one NIM-graph, up to isomorphism of NIM-reps. We shall label our graph using the subgroup $H=\{e,x\}$;
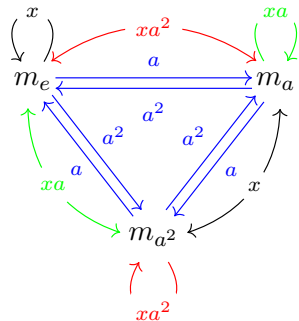
    \begin{center}\vspace{-10pt}
        \begin{tikzcd}[nodes={inner sep=0.5pt,minimum size=1ex,circle}]
            m_e \arrow[out=60,in=120,loop, "x" description,distance=1cm] \arrow[rr, leftrightarrow, "xa^2" description, red, bend left=35, shift left=1] \arrow [rr, rightarrow, "a", blue, shift left=0.5 ] \arrow[ddr, rightarrow, "a^2", blue, shift left=1] &  & m_a \arrow[out=120,in=60,loop, "xa" description,green, distance=1cm]\arrow[ll,rightarrow, "a^2",blue, shift left=1] \arrow[ddl, rightarrow, "a",blue, shift left=0.5] \arrow[ddl,leftrightarrow, "x" description, bend left=35, shift left=1] \\
            & & \\
             & m_{a^2} \arrow[uur, rightarrow, "a^2", blue, shift left=1] \arrow[uul, rightarrow, "a", blue, shift left=0.5] \arrow[uul, leftrightarrow, "xa" description, green, bend left=35, shift left=1]\arrow[out=300,in=240,loop, "xa^2" description, red,distance=1cm]&
        \end{tikzcd}
    \end{center}
    \item The trivial subgroup $H=\{e\}$. The basis elements are parameterised by the elements of $D_3$;
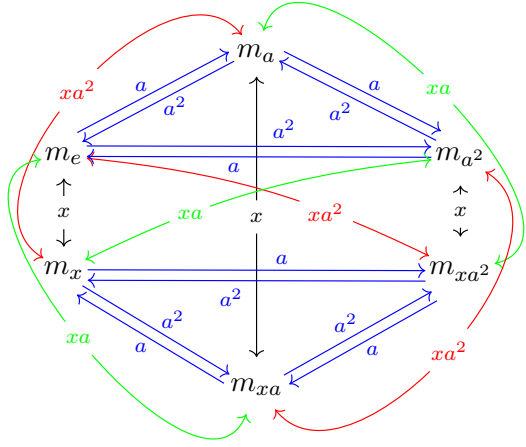
\vspace{-20pt}
    \begin{center}
        \begin{tikzcd}[nodes={inner sep=0.5pt,minimum size=1ex,circle}]
            && m_a \arrow[drr, rightarrow, "a"{yshift=-0.5ex}, blue, shift left=2] \arrow[dll, rightarrow, "a^2" {yshift=1ex}, blue, shift right=0.5] \arrow[ddd, leftrightarrow, "x" description] \arrow[ddll, leftrightarrow, "xa^2" description, red, bend right=100] \arrow[ddrr, leftrightarrow, "xa" description, green, bend left=120]&& \\
           m_e \arrow[urr, rightarrow, "a" {yshift=-0.5ex}, blue,shift left=2] \arrow[rrrr, rightarrow, "a^2"{yshift=-0.5ex, xshift=2ex} , blue, shift left=1.4] \arrow[d, leftrightarrow, "x" description] \arrow[ddrr, leftrightarrow, "xa" description, green, bend right=120]  \arrow[drrrr,leftrightarrow, "xa^2" {description, yshift=-1.8ex, xshift=5ex},red, bend left=10]  &&& & m_{a^2} \arrow[llll, rightarrow, "a"{yshift=0.5ex, xshift=-2ex} , blue, shift left=0.1] \arrow[ddll,leftrightarrow, "xa^2" description, red, bend left=100]\arrow[dllll,leftrightarrow, "xa" {description, xshift=-5ex, yshift=-1.8ex}, green, bend right=10] \arrow[ull, rightarrow, "a^2"{yshift=0.5ex} , blue, shift right=0.5] \arrow[d, leftrightarrow, "x" description] \\
           m_x \arrow[drr, rightarrow, "a^2"{yshift=-0.5ex} , blue, shift right=0.5] \arrow[rrrr, rightarrow, "a"{yshift=-0.5ex,xshift=2ex} , blue, shift left=0.1] &&& & m_{xa^2} \arrow[llll, rightarrow, "a^2"{yshift=0.6ex, xshift=-2ex} , blue, shift left=1.4] \arrow[dll, rightarrow, "a" {yshift=0.3ex,xshift=-0.2ex}, blue, shift left=2]\\
              && m_{xa} \arrow[ull, rightarrow, "a" {yshift=0.3ex,xshift=0.2ex}, blue,shift left=2] \arrow[urr, rightarrow, "a^2"{yshift=-0.5ex, xshift=0.4ex} , blue, shift right=0.5]&&
        \end{tikzcd}
    \end{center}
    
\end{itemize}

\end{example}

\subsection{Near-Group Fusion Rings}
In this subsection, we shall focus on another class of fusion rings that can be formed from a finite group $G$.

\begin{definition}[\cite{Siehler}] Let $G$ be a finite group and $\alpha \in \mZZ$. The \textit{near-group fusion ring} is the fusion ring constructed by taking the integer span of the set $G \cup \{X\}$, with multiplication of the group elements as the group operation, and with the element $X$ as:
\begin{equation*}\label{ngfusion}
\begin{split}
    X g &=g X=X, \\ 
    X^2 &=g+\alpha X \\
\end{split}
\end{equation*}
for $g \in G$. The element $X$ is self-dual, i.e $X^* = X$. This is a fusion ring, which we shall denote by $K(G,\alpha)$.
    
\end{definition}

\begin{example}\cite{EGNO}*{Example 4.10.5}
 The case $\alpha=0$ is known as the \textit{Tambara-Yamagami fusion ring}. Notice that this ring is categorifiable if and only if $G$ is abelian. 
\end{example}

The action of $K(G,\alpha)$ on a NIM-rep $(A,M)$ consists of the action of the group $G$ and the non-invertible element $X$. By the results of \Cref{sec:groupNIM}, we know that the NIM-action of the group component will correspond to some $G$-action on the NIM-rep basis. However, unlike \Cref{sec:groupNIM}, we cannot guarantee that this $G$-action is transitive on the basis $M$, due to the action of the non-invertible element $X$. This can be seen in the following example:

\begin{example}(NIM-rep over $K(\mZ_2,0)$)

Let $M=\{m^1,m^2\}$, with $\mZ_2$ acting on both elements trivially i.e. they both have a $G$-action governed by $\mZ_2$. With the action of $X$
given by $X \act m^1 = m^2, X \act m^2 = m^1$, we get a NIM-rep over $K(\mZ_2,0)$ that contains two $G$-orbits.
\end{example}

From this example, we see that the NIM-rep basis $M$ can be partitioned into $G$-orbits, which are connected to each other by the action of $X$. We will write this partition as
\begin{equation*}
    M = \bigcup_{i=1}^p \{m^i_l\}_{1\leq l\leq |G:H_i|},
\end{equation*} 
where the $i$-label counts the $p$-distinct orbits, which are governed by subgroups $\{H_i\}$. The $l$-label denotes the individual elements in each orbit.

As we have already covered the NIM-rep structure of the group part of $K(G,\alpha)$ in \Cref{sec:groupNIM}, we now need to focus on the action of the non-invertible element $X$.

\begin{proposition}\label{ngXsame}
   Let $(A,M)$ be an irreducible NIM-rep over $K(G,\alpha)$. For a fixed group orbit label $i$ in the partition of $M$, we have that $X \act m^i_{l_1} = X \act m^i_{l_2}$ for all $1 \leq l_1,l_2 \leq |G:H_i|$.
\end{proposition}
\begin{proof}
    As $m^i_{l_1}$ and $m^i_{l_2}$ are in the same $G$-orbit, there is a group element $g \in G$ such that $m^i_{l_1} = g \act m^i_{l_2}$. Then, using the module action and the fusion rules in \Cref{ngfusion}, we have that
    \begin{equation*}
        X \act m^i_{l_1} = X \act (g \act m^i_{l_2})=(Xg) \act m^i_{l_2} = X \act m^i_{l_2}.
    \end{equation*}
\end{proof}

\begin{notation}
    We shall write $c_{i,j} := (X\act m^i_{l},m^j_{k})$. It is clear from \Cref{ngXsame} that varying the $l,k$ labels has no effect. Additionally, note that by the rigidity condition of the NIM-rep,  we have $c_{i,j}=c_{j,i}$.
\end{notation} 

\begin{remark}\label{rmrk-irrepNIM}
  Irreducibility of a NIM-rep over $K(G, \alpha)$ implies that the group orbits are connected to each other by the non-invertible element $X$. 
 \end{remark}

From now on, we shall denote the action of $X$ on an element $m^i_l$ by
\begin{equation}\label{actionX}
    X \act m^i_l = \sum_{j=1}^p c_{i,j}\sum_{k=1}^{|G:H_j|} m^j_k.
\end{equation}

If we act on both sides of \Cref{actionX} with $X$, using the fusion rules in \Cref{ngfusion} along with \Cref{ngXsame}, we find that 
\begin{equation*}
    |H_i|\sum^{|G:H_i|}_{k=1}m^i_k + \alpha X \act m^i_l = \sum_{j=1}^p c_{i,j}|G:H_j| X \act m^j_l.
\end{equation*}
By counting in this equation the multiplicities of the orbit labelled with $i$, and one labelled by $q\neq i$ respectively, we get that
\begin{align}
    |H_i| + \alpha c_{i,i} = \sum^p_{j=1} c_{i,j}^2|G:H_j|\label{CBC-1} \\
    \alpha c_{i,q} = \sum_{j=1}^p c_{i,j}c_{j,q}|G:H_j|\label{CBC-2}
\end{align}
\begin{remark}
    Should our NIM-rep partition contain only one $G$-orbit, then we would only have one equation of the form of \Cref{CBC-1} as we clearly can pick no $q$ such that $q\neq i$.
\end{remark}
To classify all NIM-reps over a near-group fusion ring, we thus need to find solutions for this set of equations. To better visualise this problem, we can think of these equations in terms of matrices. By setting $C=:\{c_{i,j}\}$, the matrix that determines the action of $X$, and $B=$diag$\{|G:H_i|\}$, it is clear that the action of $X$ is governed by the matrix equation
    \begin{equation}\label{eq-CBC}
        CBC = \alpha\cdot C + |G| \cdot B^{-1}.
    \end{equation}
We are thus seeking to find choices of subgroups $\{H_i\}$ such that there is an non-negative integer-valued symmetric matrix $C$ satisfying \Cref{eq-CBC}.

As the matrix $B$ is invertible, we obtain a quadratic matrix equation in the variable $CB$; 
\begin{equation*}
   (CB)^2 = \alpha \cdot (CB) + |G|\cdot I,
\end{equation*}
where $I$ is the identity matrix. As all of the coefficient matrices commute with each other, we can solve via the quadratic equation, giving us
\begin{equation}\label{eq-CBY}
 CB = \frac{1}{2}\alpha \cdot I \pm \sqrt{(\frac{\alpha^2}{4}+|G|)\cdot I} = \frac{1}{2}\alpha \cdot I \pm  \sqrt{(\frac{\alpha^2}{4}+|G|)} \cdot Y
\end{equation}
where $Y$ is a square root of the identity matrix $I$. As all the elements in $CB$ are non-negative, the non-diagonal elements of $Y$ must all have the same sign. As both $Y$ and $-Y$ are square roots of the identity matrix, only one can provide a valid solution for $CB$, so we can always take the sign in \Cref{eq-CBY} to be positive.

\begin{remark}
    As the elements of $CB$ are integers, the non-diagonal elements of $Y$ must be divisible by $(\frac{\alpha^2}{4}+|G|)^{-1/2}$
\end{remark}

Hence, the problem of classifying NIM-reps over $K(G,\alpha)$ comes down to finding suitable square roots of identity matrices. We do this in full detail for $p=1,2$ in the following propositions. 

\begin{proposition}\label{prop-1orbitnim}
    NIM-reps over $K(G,\alpha)$ consisting of one group orbit are parametrised by pairs $(H,c_{1,1})$, where $H \subseteq G$ a subgroup and $c_{1,1} \in \mZ_{>0}$,  such that $\alpha = c_{1,1}|G:H| - \frac{|H|}{c_{1,1}}$, $c_{1,1}$ divides $|H|$ and $(c_{1,1})^2|G:H|\geq |H|$.
\end{proposition}
\begin{proof}
    As $p=1$, $Y=1$ is the only possible choice that leads to a valid solution. Let $H$ be the subgroup of $G$ that governs this orbit. The action of the non-invertible element $X$ is given by a single non-negative integer $c_{1,1} \in \mZ_{>0}$. By \Cref{nim0}, this integer is in fact strictly positive.
    
    If we switch to the element-wise notation, \Cref{eq-CBY} can be written as
    \begin{equation*}
       \sqrt{\frac{\alpha^2}{4}+|G|} = c_{1,1}|G:H| - \frac{1}{2}\alpha 
    \end{equation*}
By squaring both sides,
\begin{equation*}
     \frac{\alpha^2}{4}+|G| = (c_{1,1})^2|G:H|^2 - \alpha c_{1,1}|G:H| +\frac{\alpha^2}{4}
\end{equation*}
Rearranging this for $\alpha$ gives the required equation, with the other conditions following from the condition that $\alpha$ must be a non-negative integer.
\end{proof}

\begin{proposition}\label{prop-2orbitnim}
    All irreducible NIM-reps over $K(G,\alpha)$ consisting of two group orbits are parametrised by tuples  $(H_1,H_2,c_{1,1},c_{2,2})$, where $H_1,H_2\subseteq G$ are subgroups and $c_{1,1},c_{2,2}\in \mZ_{\geq0}$, such that $\alpha = c_{1,1}|G:H_1| + c_{2,2}|G:H_2|$, $|G|$ divides $|H_1||H_2|$, and $(\frac{|H_1||H_2|}{|G|}+ c_{1,1}c_{2,2})$ is a square number.
\end{proposition}
\begin{proof}
    It is well known that all square root $Y$s of the 2-by-2 identity matrix are either diagonal matrices whose non-zero elements are from the set $\{-1,1\}$, or have the form
    \begin{equation}\label{eq-Y2orbit}
        Y = \begin{pmatrix}
            y_{1,1} & y_{1,2} \\
            y_{2,1} & -y_{1,1}
        \end{pmatrix} \text{ where } y_{1,1}^2 + y_{1,2}y_{2,1} = 1, \quad y_{1,2},y_{2,1}\neq 0. 
    \end{equation}
It is clear by \Cref{rmrk-irrepNIM} that irreducible NIM-reps can only come from square root matrices of the form in \Cref{eq-Y2orbit}.

Let $H_1,H_2$ be subgroups of $G$ that govern the group orbits in the NIM-rep, and $C=\{c_{i,j}\}_{1\leq i,j\leq 2}$ the matrix governing the NIM-action of $X$.
By inputting this data into \Cref{eq-CBY}, we get the following system of equations:
\begin{align*}
    &c_{1,1}|G:H_1| = \frac{1}{2}\alpha + \sqrt{\frac{\alpha^2}{4}+|G|}y_{1,1}\\
    &c_{2,2}|G:H_2| = \frac{1}{2}\alpha - \sqrt{\frac{\alpha^2}{4}+|G|}y_{1,1}\\
    &c_{1,2}|G:H_2| = \sqrt{\frac{\alpha^2}{4}+|G|}y_{1,2}\\
    &c_{1,2}|G:H_1| = \sqrt{\frac{\alpha^2}{4}+|G|}y_{2,1}
\end{align*}

By adding the first two equations, we get that $\alpha = c_{1,1}|G:H_1|+c_{2,2}|G:H_2|$. 
By multiplying the first two equations together and the last two equations together, we obtain
\begin{align*}
    &c_{1,1}c_{2,2}|G:H_1||G:H_2| = \frac{\alpha^2}{4}-\Bigg(\frac{\alpha^2}{4}+|G|\Bigg)(y_{1,1})^2,\\
    &(c_{1,2})^2|G:H_1||G:H_2| = \Bigg(\frac{\alpha^2}{4}+|G|\Bigg)y_{1,2}y_{2,1}.    
\end{align*}
By using the defining relation in \Cref{eq-Y2orbit}, and rearranging to solve for $(c_{1,2})^2$, we find that
\begin{align*}
    (c_{1,2})^2= \frac{|H_1||H_2|}{|G|} + c_{1,1}c_{2,2}.
\end{align*}
Hence we can always obtain $c_{1,2}$ from the other input data.
The remaining conditions come from the fact that $c_{1,2}$ must be a positive integer.

Hence, a two-orbit NIM-rep over $K(G,\alpha)$ has input data of $(H_1,H_2,c_{1,1},c_{2,2})$ satisfying the above conditions, and explicit $X$-action given by 
\begin{align*}
    X \act m^1_i = c_{1,1}\sum_{k=1}^{|G:H_1|}m^1_k + \sqrt{\frac{|H_1||H_2|}{|G|} + c_{1,1}c_{2,2}}\sum_{k=1}^{|G:H_2|}m^2_k,\\
     X \act m^2_i =  \sqrt{\frac{|H_1||H_2|}{|G|} + c_{1,1}c_{2,2}}\sum_{k=1}^{|G:H_1|}m^1_k + c_{2,2}\sum_{k=1}^{|G:H_2|}m^2_k .
\end{align*}
\end{proof}

While this only completes the classification for NIM-reps consisting of two group orbits, this is sufficient to completely classify irreducible NIM-reps over the Tambara-Yamagami fusion rings.

\begin{corollary}\label{cor-TY}
    All irreducible NIM-reps over the Tambara-Yamagami fusion ring $K(G,0)$ consist of at most two group orbits.
\end{corollary}
\begin{proof}
   By setting $\alpha = 0$ in \Cref{eq-CBC}, we easily see that the matrix $CBC$ must be diagonal. Element-wise, this means that 
   \begin{equation*}
       (CBC)_{i,j} = \sum_{k=1}^pc_{i,k}c_{k,j}|G:H_i| = 0,
   \end{equation*}
   for all $i \neq j$. However, by \Cref{rmrk-irrepNIM}, there always exist a choice of $i\neq j$ such that both $c_{i,k},c_{i,j}\neq 0$ when $p\geq 3$. Thus, as $C$ is symmetric, the result follows.
\end{proof}
    
\begin{example}(NIM-reps over the Ising fusion ring)

Recall \Cref{ising}. The Ising fusion ring can be viewed as the Tambara-Yamagami fusion ring $K(\mZ_2,0)$. So by \Cref{cor-TY}, we only need to check for 1 and 2-orbit NIM-reps.

When $p=1$, we must have $c_{1,1}|\mZ_2:H| = |H|/c_{1,1}$ by \Cref{prop-1orbitnim}. But as the only choices for $H$ are the trivial group and $\mZ_2$, both of which lead to non-integer values for $c_{1,1}$, we see that there are no 1-orbit NIM-reps over $K(\mZ_2,0)$.

For $p=2$, by \Cref{prop-2orbitnim} we must have $0=c_{1,1}|G:H_1| + c_{2,2}|G:H_2|$, which only occurs when $c_{1,1}=c_{2,2}=0$. Hence, we see that $(c_{1,2}^2) = \frac{|H_1||H_2|}{|G|}$, which must be an square number.
The only possible choice of subgroups is $H_1\cong\{1\},H_2 \cong \mZ_2$ (or vice versa, which gives an equivalent NIM-rep). In this case, $c_{1,2}=1$. 

Thus, there is only one irreducible NIM-rep over the Ising fusion ring, given by the tuple $(\{1\},\mZ_2,0,0)$.
\end{example}

For $p\geq 3$, there is not an explicit classification of square roots of the identity matrix. For instance, it is still unknown whether there exists a Hadamard matrix of order $4k$ for every positive integer $k$. So, in what follows we will detail some relevant examples.

\begin{example}
    A Hadamard matrix of order $n$ satisfies $H_nH_n = nI_n$, and rows are mutually orthogonal and the matrix contains only $+1$ and $-1$, so clearly $\frac{1}{\sqrt{n}}H_n$ is a square root of the identity. However, $H_n$ contains exactly $n(n-1)/2$ elements that are $-1$, which contradicts the fact all non-diagonal elements of $Y$ must be non-negative, of which there are $n(n-1)$. So the only value valid is $n=1$, which is the trivial case.
\end{example}

\begin{example}(NIM-rep over $K(\mZ_q,q-1), q$ prime)

    Consider the $p$-by-$p$ matrix that has $-\frac{1}{2}\alpha$ as its diagonal elements, and $1$ for all off-diagonal elements;
    \begin{equation}\label{ex-Zp}
        Z_p = \begin{pmatrix}
            -\frac{1}{2}\alpha & 1 & \hdots &1\\
            1 & \ddots & \ddots & \vdots\\
            \vdots & \ddots  & \ddots&1\\
            1 & \hdots & 1& -\frac{1}{2}\alpha
        \end{pmatrix} 
    \end{equation}
If there is a NIM-rep that corresponds to this matrix, it must have $c_{i,i}=0$ for all i by looking \Cref{eq-CBY}
    
We want to find values of $\alpha$ and $G$ such that this matrix is a square root of $(\frac{\alpha^2}{4}+|G|)\cdot I_p$. It is easily verified that that
\begin{equation*}
    (Z_p^2)_{i,i} = \frac{\alpha^2}{4} + p-1, \quad (Z_p^2)_{i,j} = p-2-\alpha, \; i\neq j.
\end{equation*}
Thus, the matrix $Z_p$ may give a NIM-rep via \Cref{eq-CBY} only if $\alpha = p-2$ and $|G| = p-1$. Thus $Z_p$ may give a NIM-rep over $K(G,p-2)$, where $|G|=p-1$. The existence of such a NIM-rep depends on the particular choice of group, but we can provide the case such that $p=q+1$, where $q$ is prime.

The only group of order $q$ is the cyclic group $\mZ_q$, so $Z_{q+1}$ defines a NIM-rep only over $K(\mZ_q,q-1)$. By putting \Cref{ex-Zp} into \Cref{eq-CBY}, it is clear that $c_{i,j}|G:H_j| = 1$ for all $i\neq j$. As the only subgroups of $\mZ_q$ are the trivial subgroup and the group itself, the only choice that leads to an integer value of $c_{i,j}$ is if $H_j \cong \mZ_q$ for all $j$. This gives $c_{i,j}=1$. 

Thus, for $q$ prime, we have a NIM-rep $(A,M)$ over $K(\mZ_q,q-1)$ where the $M$ consists of $q+1$ elements $\{m^{i}\}_{1\leq i\leq q+1}$, the $\mZ_q$-action on $M$ is trivial, and the action of the non-invertible element $X$ is given by
\begin{equation*}
    X \act m^i = \sum_{\substack{j=1 \\ j\neq i}}^{q+1}m^j.
\end{equation*}
    
\end{example}

Note that to look at the possible case that a NIM-rep over $K(G,\alpha)$ can be partitioned into three or more orbits, one would need to find symmetric integer matrix solutions to the matrix equation. We can use the following result to reduce the options in some circumstances.
\begin{proposition}
    A NIM-rep over $K(G,\alpha)$ consists of an odd number of group orbits if and only $\sqrt{\frac{\alpha^2}{4}+|G|}$ is a rational number. 
\end{proposition}
\begin{proof}
    As the matrix $CB$ contains only integers, from \Cref{eq-CBY} we see that the matrix $\sqrt{(\frac{\alpha^2}{4}+|G|)} \cdot Y$ contains only rational numbers. Its determinant is then rational. But $Y$ is a square root of the identity matrix, that has determinant $\pm 1$. Hence, the determinant of $\sqrt{(\frac{\alpha^2}{4}+|G|)} \cdot Y$ is equal to $(\frac{\alpha^2}{4}+|G|)^{p/2}$. The result is then immediate, as the rational numbers are closed under multiplication, so this determinant is rational if and only if  $(\frac{\alpha^2}{4}+|G|)^{1/2}$ is rational.
\end{proof}

We can also express $\alpha$ in terms of elements of $C$ and $B$:
\begin{corollary}
    Suppose $(A,M)$ is an irreducible NIM-rep over $K(g,\alpha)$ that can be partitioned into $p$ $G$-orbits.
    \begin{enumerate}
        \item[] If $p=1$, then $\alpha= c_{1,1}|G:H_1| - |H_1|/c_{1,1}$.
        \item[] If $p\geq 2$. Then $\alpha = \sum_{j=1}^p \frac{c_{i,j}c_{j,q}}{c_{i,q}}|G:H_j|$ for any orbit labels $i\neq q$ where $c_{i,q}>0$.  
    \end{enumerate}
\end{corollary}
\begin{proof}
    For the $p=1$ case, we simply rearrange \Cref{CBC-1} for $\alpha$. 
    The $p=2$ case is just as immediate, instead  rearranging \Cref{CBC-2}. Using the fact $c_{i,1}$ is non-zero by \Cref{rmrk-irrepNIM}, we always have at least one choice of matrix elements to write $\alpha$ in this form.
\end{proof}

We conclude this section with some concrete examples.

\begin{example} 

An example of such a NIM-rep is given by the tuple $(\mZ_{75},\mZ_{75},5,5)$ over the fusion ring $K(\mZ_{75}, 10)$. As the two subgroups that govern the partition of the NIM-rep are the whole group, these orbits consist of one NIM-element each, so $M=\{m,n\}$.

By the found formulas, $c_{1,1} = c_{2,2} = 5, c_{1,2} = 10$ and so the NIM-rep has the following structure;
\begin{align*}
    g \act m = m \quad &, \quad g \act n = n & \forall g \in \mZ_{75},\\
    X \act m = 5m + 10 n \quad &, \quad X \act n = 10m + 5n. &
\end{align*}
\end{example}

\begin{example}
    $K(\mZ_{175},62)$ has a NIM-rep corresponding to $(\mZ_{35}, \mZ_{25}, 11,1)$. Labelling the two group orbits by $\{m_i\}_{1\leq i \leq 5}$, ${n_j}_{1\leq j\leq 7}$, the explicit structure of the non-invertible element is given by 
    \begin{align*}
        X \act m_i = 11\sum_{l=1}^5 m_l + 4 \sum_{k=1}^7 n_k \quad &, \quad X \act n_j  = 4 \sum_{l=1}^5m_l + \sum_{k=1}^7 n_k .
    \end{align*}
\end{example}

\subsection{Admissible NIM-Reps and Algebra Objects for Pointed and Near-Group Fusion Categories}\label{sec:admiss}

Finally, we extract algebra objects from the NIM-reps we have computed. While we cannot read out explicit algebra morphisms, we still are able to recover a collection of familiar results.

\begin{proposition}{(Admissible NIM-reps for group rings)}
    All irreducible NIM-reps over a group fusion rings $R(G)$ are admissible.
\end{proposition}
\begin{proof}
   Let $M(H)$ be a NIM-rep over $R(G)$. Recall that the basis elements are parametrised by the left cosets $G/H$. For any two NIM-basis elements $m_{g_i}$, $m_{g_j}$, we can have that $(g_jg_i^{-1})\act m_{g_i} = m_{g_j}$. Hence any NIM-basis element can take the role of $m_0$ in \Cref{prop-NIMadmis}, thus $M(H)$ is admissible.
\end{proof}
\begin{remark}
The algebra object associated to a NIM-rep $M(H)$ over $R(G)$ is then given by $\bigoplus_{h \in H}b_h$. It is clear from \Cref{ex-d3,ex-klein4} that this object can be seen by counting the self-loops present in the NIM-graph. This agrees with the classification of algebras in pointed categories given in \cites{Nat,Ost}.
\end{remark}

\begin{proposition}
    An admissible irreducible NIM-rep over a near-group fusion ring consists of either one group orbit, parameterised by $(H,c_{1,1})$, or two group orbits, parametrised by $(H_1,H_2,0,c_{2,2})$.
    \end{proposition}
\begin{proof}
For a NIM-rep consisting of one orbit, we can argue in the same way as \Cref{prop-NIMadmis} that any NIM-basis element can be set as $m_0$ in \Cref{prop-NIMadmis}.

For a NIM-rep consisting of more than one group orbit, the only way to have an $m_0$ that connects to the orbits it does not lie in is via the  non-invertible element $X$. So for $n \in M$, where $n$ is not in the same orbit as $m_0$, we must have $X \act m_0= n$. But then by the fusion rules
    \begin{equation*}
      X \act n = X^2 \act m_0 = \sum_{i \in G\cdot m_0} m_i + \alpha X \act m_0 = \sum_{i \in G\cdot m_0} m_i +\alpha n  
    \end{equation*}
From this we see that the group orbits of $m_0$ and $n$ are only connected to each other, so by the irreducibility of the NIM-rep it only contains two group orbits.

If we label the orbit containing $m_0$ by 1, and the one containing $n$ by 2, we note that the condition $X \act m_0 = n$ implies $c_{1,1} = 0$. This gives the result.
\end{proof}
\begin{remark}
We thus have two forms of algebra objects arising in categories associated to near-group fusion rings.
\begin{itemize}
    \item For a NIM-rep paramatrised by $(H,c_{1,1})$, the algebra is given as an object by $\oplus_{h \in H}b_h \oplus c_{1,1}X$.
    \item For a NIM-rep parametrised by $(H_1,H_2,0,c_{2,2})$, the algebra object is given by $\oplus_{h \in H_1}b_h$, i.e it is of the form of a group algebra object.
\end{itemize}
\end{remark}

Algebra objects representing module categories over near-group categories have been studied previously at \cites{MM,Gal}: in the case of non-group theoretical Tambara-Yamagami categories with abelian $G$ we have twisted group algebras, see \cite[Proposition 5.7]{Gal}. \cite[Section 8 and 9]{MM} proceed in a slightly more general way, for $G$-graded fusion categories with $G$ not necessarily abelian. Only for the case of Tambara-Yamagami (see \cite[Section 9]{MM}) they recover the same two families of algebras we observe.

\subsection{$(A_1,l)_{\frac{1}{2}}$ Fusion Rings and its Admissible NIM-Reps and Algebra Objects} 
Following \cite{NWZ} (other useful references are \cites{EP,FK}), we can construct a modular tensor category $(A_1,l)$ from a quantum group of type $A_1$ at level $l \in \mZ_{+}$. The Grothendieck ring $\mathrm{Gr}((A_1,l))$ has basis $\{V_i\}_{i \in [0,l]}$, and the fusion coefficients of $V_iV_j = \sum_{k=0}^l c_{i,j}^k V_k$ are given by 
\begin{equation}\label{eq-quantumfusion}
    c_{i,j}^k = \begin{cases}
        1, & \text{if } |i-j| \leq k \leq \mathrm{min}(i+j,2k-i-j) \text{ and } k \equiv i+j \;\mathrm{mod}\; 2,\\
        0, & \mathrm{else}
    \end{cases}
\end{equation}
We note here that, as seen by the fusion rules, $\mathrm{Gr}((A_1,l))$ is commutative.
\begin{definition}\label{def-length}
    For an object $V_i \in \mathrm{Gr}((A_1,l))$, we define the \textit{length} of $V_i$ to be $\mathrm{length(V_i):= \sum\limits_{k=0}^{l}}c_{i,j}^k.$
\end{definition}

When $l$ is a positive odd integer, we can define a modular subcategory $(A_1,l)_{\frac{1}{2}}$ by taking the full subcategory with simple objects 
\begin{equation*}
    \mathrm{Irr}((A_1,l)_{\frac{1}{2}})= \Bigg\{ V_{2i}\; | \; 0 \leq i \leq \frac{-1}{2} \Bigg \}.
\end{equation*}
\begin{remark}\label{rmk-length}
    In the fusion ring $\mathrm{Gr}((A_1,l)_{\mathrm{pt}})$, it is simple to check using the fusion rules that $V_i \neq V_j$ if and only if $\mathrm{length}(V_i) \neq \mathrm{length}(V_j)$.
\end{remark}

We shall now focus on finding the admissible NIM-reps of $\mathrm{Gr}((A_1,l)_{\frac{1}{2}})$. As $l=1$ results in the trivial ring, we shall assume $l\geq 3$. Recall from \Cref{prop-NIMadmis} that an admissible NIM-rep $(A,M)$ over a fusion ring has a distinguished basis element $m_0$. We have seen in \Cref{sec:admiss} that in the case of group and near-group fusion rings, we can have objects in the ring basis $b_i,b_j$ such that $b_i \act m_0 = b_j \act m_0$, due to the invertibility of the group parts of these fusion rings. This is not the case when working with $\mathrm{Gr}((A_1,l)_{\frac{1}{2}})$.
\begin{lemma}\label{lem-noninvert}
    Let $(A,M)$ be an admissible NIM-rep over $\mathrm{Gr}((A_1,l)_{\frac{1}{2}})$, and $V_i \neq V_j \in \mathrm{Gr}((A_1,l)_{\frac{1}{2}})$ such that $V_i \act m_0, V_j \act m_0 \in M$. Then $V_i \act m_0 \neq V_j \act m_0$.
\end{lemma}
\begin{proof}
    If we assume that $V_i \act m_0 = V_j \act m_0 $, then as the fusion ring is commutative, we have that $V_i^2 \act m_0 = V_j^2 \act m_0$. But by \Cref{rmk-length}, $\mathrm{length}(V_i) \neq \mathrm{length(V_j)}$, and so the only way that $V_i^2 \act m_0 = V_j^2 \act m_0$ is if some $V_k$ in the decomposition of $V_i^2$ or $V_j^2$ (whichever has larger length), satisfies $V_k \act m_0 = 0$. This has been ruled out by \Cref{nim0}.
\end{proof}

  \begin{remark}\label{rmk-dimadmiss}
     If we have objects $V_i,V_j \in \mathrm{Gr}((A_1,l)_{\mathrm{pt}})$ such that $\mathrm{length}(V_i) > \mathrm{length}(V_j)$, then it is easily verified using the fusion rules in \Cref{eq-quantumfusion} that
     $(V_i^2 \act m_p,m_p) > (V_j^2 \act m_p,m_p)$. Thus, if $V_i \act m_p \in M$, then we immediately have that $V_j \act m_p \in M$. In the case of an admissible NIM-rep $(A,M)$ over $\mathrm{Gr}((A_1,l)_{\mathrm{pt}})$, if the basis $M$ has cardinality d, we immediately get that the objects $V_j$ that satisfy $V_j \act m_0 \in M$ are exactly those of $\mathrm{length}(V_j)\leq d$.
  \end{remark}

\begin{proposition}\label{prop-pqlessthan1}
    In any NIM-rep $(A,M)$ over the fusion ring $(A_1,l)_{\frac{1}{2}}$, $(V_{l-1} \act m_p,m_q) \leq 1$ for all $m_p, m_q \in M$.
\end{proposition}
\begin{proof}
If we assume that $(V_{l-1}\act m_p,m_q)= a_{l-1,p}^q\geq 2$, we can use the fusion rules in \Cref{eq-quantumfusion} to obtain
\begin{equation*}
    m_p + V_2 \act m_p = V_{l-1}^2 \act m_p = a_{l-1,p}^q V_{l-1} \act m_q + \sum_{\substack{k \in L\\ k \neq q}} a_{l-1,p}^k V_{l-1} \act m_k
\end{equation*}
Applying the form $(-,m_p)$, and using the rigidity condition of the NIM-rep, we find that
\begin{equation}\label{eq-V2greater3}
    (V_2 \act m_p,m_p) \geq a_{l-1,p}^q(V_{l-1}\act m_q,m_p) -1 \geq 3
\end{equation}

The fusion rules in \Cref{eq-quantumfusion} give that, when $l \geq 3$,
$  V_{2j} V_2 = V_{2j-2} + V_{2j}+V_{2j+2},$ when $ 1\leq j \leq  \frac{l-3}{2},$ and $
    V_2 V_{l-1} = V_{l-3} + V_{l-1}. $
We let $h_{i,p} := \sum_{k\in L}a^k_{i,p}$, which counts the number of NIM-basis elements in the decomposition of $V_i \act m_p$.  Applying the fusion rules to $V_2jV_2 \act m_p= \sum_{k \in L}a_{2,p}^kV_2j \act m_k$, we obtain
\begin{align*}
    V_{2j-2}\act m_p + V_{2j+2} \act m_p &= (a_{2,p}^p-1)V_{2j} \act m_p + \sum_{\substack{k \in L\\ k \neq p}} a_{2,p}^k V_{2j} \act m_k , \quad 1 \leq j \leq \frac{l-3}{2} \\
    V_{l-3} \act m_p &= (a_{2,p}^p-1)V_{l-1}\act m_p + \sum_{\substack{k \in L \\ k \neq p}} a_{2,p}^k V_{l-1} \act m_k
\end{align*}
By noting that $h_{i,p}>1$ for all choices of $i,p$, and $a_{2,p}^p\geq 3$ by \Cref{eq-V2greater3}, when we count the NIM-basis elements on each side we obtain the following inequalities;
\begin{align*}
    &h_{2j-2,p}+h_{2j+2,p}\geq 2 h_{2j,p}+ h_{2,p} -3, \quad 1\leq j\leq \frac{l-3}{2},\\ &h_{l-3,p} \geq h_{l-1,p} + h_{2,p} -3
\end{align*}

By taking the inequality for each $1\leq j \leq \frac{l-1}{2}$ and summing them together, we obtain
\begin{align*}
    1 + h_{2,p} + 2h_{4,p} + ... + 2h_{l-3,p} + h_{l-1,p} &\geq 2h_{2,p} + 2h_{4,p} + ...+ 2h_{l-1,p} + \frac{l-1}{2}(h_{2,p}-3) \\
    \implies \frac{3l-1}{2} &\geq \frac{l+1}{2}h_{2,p} +h_{l-1,p} \geq \frac{l+1}{2}h_{2,p} + 2\\
    \implies \frac{3l-5}{l+1}&\geq h_{2,p}
\end{align*}
where the last inequality in the second line follows due to our initial assumption.
However, the last inequality cannot be satisfied as the fraction on the left-hand side is strictly less than 3 for all values of l, which contradicts $(V_2 \act m_p,m_p)= a_{2,p}^p \geq 3$. Hence we have a contradiction, and so, $(V_{l-1}\ast m_p,m_q) \leq 1$, for all $m_p,m_q \in M$.
\end{proof}
\begin{lemma}\label{lem-V(k-1)}
    In any NIM-rep $(A,M)$ over $(A_1,l)_{\frac{1}{2}}$, $(V^2_{l-1}\act m_p,m_p) \leq 3$ for all $m_p,m_q \in M$.
\end{lemma}
\begin{proof}
    If we assume that $ (V^2_{l-1}\act m_p,m_p) > 4$, then by \Cref{prop-pqlessthan1}, we have that $h_{l-1,p}= (V^2_{l-1}\act m_p,m_p)>4$. By the fusion rules in \Cref{eq-quantumfusion}, we have that $(V_2 \act m_p,m_p)\geq 3$. We are in a very similar setup to the proof of \Cref{prop-pqlessthan1}, which if we follow through results in the inequality, we get
    \begin{equation*}
        \frac{3l-9}{l+1}\geq h_{2,p}.
    \end{equation*}
    This fraction is also strictly less than 3, so we obtain the desired contradiction.
\end{proof}
\begin{proposition}\label{prop-leq2}
    Let $(A,M)$ be an admissible NIM-rep over $(A_1,l)_{\frac{1}{2}}$. Then there exists no $m_k\in M$ such that $(V^2_{l-1} \act m_k,m_k) = 3$.
\end{proposition}
\begin{proof}
    Assume there exists some $m_k \in M$ such that $(V^2_{k-1} \act m_k,m_k) = 3$. By \Cref{lem-V(k-1)}, we can write 
    \begin{equation*}
        V_{k-1}\act m_k = m_x + m_y + m_z,
    \end{equation*}
    where $m_x, m_y, m_z \in M$ are distinct NIM-basis elements. The fusion rule of $V_{l-1}^2$ gives us that $(V_2 \act m_k,m_k)=2$.

    As the NIM-rep is admissible and $m_x,m_y,m_y$ are distinct NIM-basis elements, the cardinality of the NIM-basis $M$ is at least $3$, so by \Cref{rmk-dimadmiss} we know that $V_2 \act m_0, V_{l-1} \act m_0 \in M$. We also know that there exists a $V_j \in  (A_1,l)_{\frac{1}{2}}$ such that $V_j \act m_0 = m_k$. Using the fusion rules, we find that 
    \begin{equation*}
        V_{l-1} \act m_k = V_{l-1}V_j \act m_0 = V_{l-1-j} \act m_0 + V_{l+1-j} \act m_0 .
    \end{equation*}
    Acting with $V_{l-1}$ again on both sides and then using the form $(-,m_k)$, we have
    \begin{align}
        m_k + V_2 \act m_k &= V^2_{k-1} \act m_0 = V_j \act m_0 + 2 V_{j+2} \act m_0 + V_{j+4} \act m_0 \nonumber \\
       \implies  V_2 \act m_k &=  2V_{j+2} \act m_0 + V_{j+4} \act m_0 \label{eq-quantum1}.
    \end{align}
    A second way to calculate $V_2 \act m_k$ is as follows:
    \begin{equation} \label{eq-quantum2}
        V_2 \act m_k = V_2V_j \act m_0 = V_{j-2} \act m_0 + V_j \act m_0 + V_{j+2} \act m_0.
    \end{equation}
    By applying the form $(-,m_k)$ to both \Cref{eq-quantum1} and \Cref{eq-quantum2}, we have that
    \begin{align}
        &2 = 2(V_{j+2} \act m_0,m_k) + (V_{j+4} \act m_0,m_k) \label{eq-quantum3}\\
        &1 = (V_{j-2} \act m_0,m_k) + (V_{j+2} \act m_0,m_k) 
    \end{align}
    By the fusion rules \Cref{eq-quantumfusion} and \Cref{rmk-dimadmiss}, it is clear that we can only satisfy \Cref{eq-quantum3} when $j \leq \frac{l\pm 1}{2}$, where the sign is determined by $l \equiv \mp 1 \text{ mod } 4$. But then we immediately have that $(V_{j-2} \act m_k, m_1) = 0$ by \Cref{rmk-dimadmiss}, so the above equations give us that 
    \begin{equation} \label{eq-quantum5}
        (V_{j+2} \act m_0,m_k) = 1, \; \text{ and } (V_{j+4} \act m_0, m_k)=0.
    \end{equation}
    If we calculate the fusion rules, we see that the only term that appears in $(V_{j+2}V_j \act m_0, m_0)$ and not in $(V_{j+4})$ is $(V_2 \act m_0,m_k)$. But as $V_2 \act m_0 \in M$, we have by \Cref{lem-noninvert} that this term is non-zero. Hence we can never satisfy \Cref{eq-quantum5}, giving a contradiction. Thus we must have $(V_{l-1}^2 \act m_0,m_0) \leq 2$ for all $m_k \in M$.
\end{proof}
\begin{proposition}\label{prop-quantumNIM}
    Up to isomorphism, there is only one admissible NIM-rep over $A(1,l)_{\frac{1}{2}}$.
\end{proposition}
\begin{proof}
Suppose we have an admissible NIM-rep $(A,M)$ over $A(1,l)_{\frac{1}{2}}$ where $M$ has cardinality $d$. Then there exists a $V_j \in A(1,l)_{\frac{1}{2}}$ where $\mathrm{length(V_j) = d}$ and $V_j \act m_0 \in M$. Acting with $V_{l-1}$ we get that 
\begin{equation*}
    V_{l-1}V_j \act m_0 = V_{l-1-j}\act m_0 + V_{l+1-j} \act m_0.
\end{equation*}
Using the fusion rules \Cref{eq-quantumfusion} and \Cref{def-length}, if $j < \frac{k\pm 1}{2}$, (where the sign depends on $k \equiv \mp 1 \text{ mod } 4$), then length$(V_{l+1-j})= d+1$ and so $V_{l-1-j} \act m_0 \notin M$, by \Cref{rmk-dimadmiss}. Similarly, if $j > \frac{k\pm 1}{2}$, then $V_{l+1-j} \act m_0 \notin M$. In both cases, this leads to
\begin{equation*}
    (V^2_{l-1} V_{j} \act m_0, V_j \act m_0) \geq (V_{l-1-j} \act m_0,V_{l-1-j} \act m_0 ) + (V_{l+1-j} \act m_0,V_{l+1-j} \act m_0 ) = 3.
    \end{equation*}
    This  contradicts $V_{j} \act m_0 \in M$ by \Cref{prop-leq2}, and so we must have that $j = \frac{k\pm 1}{2}$. All other objects $V_i$ have length less than $V_j$, and so every $m_k \in M$ is of the form $V_i \act m_0 \in M$. This gives that full NIM-rep structure is simply generated by the fusion rules.
\end{proof}

Now that we have all admissible NIM-reps over $\mathrm{Gr}((A_1,l)_{\frac{1}{2}})$, we turn our attention back to $(A_1,l)$. We can form a second full subcategory $(A_1,l)_{\mathrm{pt}}$ by taking the simple objects  $\mathrm{Irr((A_1,l)_{pt})}=\{V_0,V_l\}$. 
\begin{lemma}[\cite{NWZ} Section 4]
    There is an equivalence of modular tensor categories $(A_1,l) \simeq (A_1,l)_{\frac{1}{2}} \boxtimes (A_1,l)_{\mathrm{pt}}$
\end{lemma}
On the level of Grothendieck rings, this equivalence tells us that the basis set of $\mathrm{Gr((A,l))}$ is equivalent to the direct product of basis sets of $\mathrm{Gr}((A_1,l)_{\frac{1}{2}})$ and $\mathrm{Gr}((A_1,l)_{\frac{1}{2}}) = R(\mZ_2) $, with the ring structure induced component-wise from the ring structures of $\mathrm{Gr}((A_1,l)_{\frac{1}{2}})$ and $R(\mZ_2)$ respectively. 

If we have a NIM-rep $(A,M)$ over $\mathrm{Gr((A,l))}$, we immediately gain a NIM-rep over both $\mathrm{Gr}((A_1,l)_{\frac{1}{2}})$ by restricting along the natural embedding. NIM-reps over the Grothendieck ring $\mathrm{Gr}((A_1,l))$ have been classified in \cite{EK} and are in one-to-one correspondence with simply laced Dynkin diagrams with Coxeter number h=l+2. In the case that $l$ is an odd integer, this gives the only NIM-rep as $\mathrm{Gr}((A_1,l))$ viewed as a NIM-rep over itself. The restriction of this NIM-rep to $\mathrm{Gr}((A_1,l)_{\frac{1}{2}})$ is exactly the single admissible NIM-rep found in \Cref{prop-quantumNIM}.

\newpage
\appendix
\section{Modular invariants and the exponents conjecture for low rank modular tensor categories, joint with Devi Young}
\label{modularinvs}

Here we would like to test a conjecture that relates NIM-reps of modular tensor categories and their so-called modular invariants. Let us first introduce several notions.

\begin{definition}
    Let $\cC$ be a modular tensor category of rank $k$ with modular data $S$ and $T$. A \textit{modular invariant} is a squared matrix $\mathcal{Z}$, with rows and columns labeled by $\mathrm{Ob} \left( \cC \right)$, satisfying:
    \begin{enumerate}
        \item[MI1:] $\mathcal{Z} S=S \mathcal{Z}$ and $\mathcal{Z} T=T \mathcal{Z}$;
        \item[MI2:] $\mathcal{Z}_{\mathrm{ab}} \in \mathbb{N}$, $\forall a,b \in \mathrm{Ob} \left( \cC \right)$; and
        \item[MI3:] $\mathcal{Z}_{00}=1$.
    \end{enumerate}
\end{definition}

The easiest example of a modular invariant is the identity matrix. Indeed, the most interesting ones would be different ones than this one.

\begin{definition}
    Given a modular invariant $M$ associated to a modular tensor category $\cC$ of rank k with $\mathrm{Ob} \left( \cC \right)=\lbrace \one, X_1,X_2,\ldots,X_{k-1} \rbrace$, the \textit{exponent} of $M$ is the multi-set $\mathcal{E}_M$ where $a \in \mathrm{Ob} \left( \cC \right)$ appears with multiplicity $M_{\mathrm{aa}}$. We will denote the exponent as:  $\mathcal{E}_M=\left( \one^{M_{00}},X_1^{M_{11}},\ldots,X_{k-1}^{M_{k-1 k-1}} \right)$.
\end{definition}

Given a certain NIM-rep $N$ in a modular tensor category, observed in the sense of \cites{Gannon,Gannonbook} described in \Cref{rmkNIMreps}, the matrices $N_a$ can be simultaneously diagonalised by a unitary matrix. Each eigenvalue of $N_a$ equals $S_{a,b}/S_{0,b}$, for some $b \in \mathrm{Ob} \left( \cC \right)$. Each eigenvalue corresponds to at most one element in $\mathrm{Ob} \left( \cC \right)$.  This allows us to introduce a second, seemingly different notion of an exponent attached to the modular data:

\begin{definition}
    Let $N$ be a NIM-rep associated to a modular tensor category of rank k with $\mathrm{Ob} \left( \cC \right)=\lbrace \one, X_1,X_2,\ldots,X_{k-1} \rbrace$. The multiset associated to the assignment $a \mapsto N_a$ is defined to be $\mathcal{E}_a \left( N \right) = \left (\one^n_{0}, X_1^{n_1}, ..., X_{k-1}^{n_{k-1}} \right)$, where  $n_b \in \{0,1\}$ is the multiplicity of $S_{a,b}/S_{0,b}$ corresponding to an eigenvalue in $N_a$. This multiset is independent of $a$, so we define the \textit{exponent} of $N$ to be $\mathcal{E}\left( N \right) := \mathcal{E}_a \left( N \right) $.
\end{definition}

These two notions seem to be related in the following way:
\begin{conjecture}
    Consider a rational conformal field theory described by a modular tensor category $\cC$. Then, for every modular invariant $M$ there is a NIM-rep $N$ satisfying that: \begin{equation*}
        \mathcal{E} \left( N \right)=\mathcal{E}_M.
    \end{equation*}
\end{conjecture}

There is certain evidence that this conjecture is not true in general, see e.g. \cite{GannonNIM}. Still, it is an interesting question to see up to which point it holds, and why it is the case. In what follows we test this conjecture for all unitary modular tensor categories of rank less or equal 4 as classified in \cite{RSW}. For this, we use the NIM-reps we have computed in the present article since they all fall into the families we have classified. Tables \ref{MIrank2MTCs}, \ref{MIrank3MTCs}, \ref{MIrank4MTCs-1}, \ref{MIrank4MTCs-2}, and \ref{MIrank4MTCs-3} contain the modular invariants and NIM-reps associated to each of these categories, as well as their respective associated exponents.

Here, note that the fusion rings of the Ising and $A\left(1,2 \right)$ MTCs, and the fusion rings of the toric and $D\left(1, 4 \right)$ MTCs, are the same. As the NIM-reps do not carry any of the modular data beyond this, the NIM-reps for each of these pair of categories are the same. For the Ising and $A \left(1,2\right)$ pair, the remaining modular data is so similar that we cannot even distinguish them using their modular invariants. This is not the case for the toric and $D\left(1, 4 \right)$ pair.

What we see is that for rank 2,3 MTCs each modular invariant has a corresponding NIM-rep, and the correspondence is true. However, at rank 4 it already starts breaking down: the $\mZ_4$ MTC has a NIM-rep that does not correspond to a modular invariant, and the toric and $D(4,1)$ MTCs each having two modular invariants associated to a single NIM-rep.


\begin{table}[H]
    \centering
    \begin{tabular}{c|c|c||c|c}
         \textbf{Category} & \textbf{Modular} & \textbf{Exponent} & \textbf{NIM-rep} & \textbf{Exponent} \\
         & \textbf{invariant} &&& \\ \hline
         Semion & $\mathrm{Id}$ & $\left( \one^{1},X_1^{1} \right)$ & $X_1 \mapsto \left( \begin{matrix}
             0 & 1 \\ 1 & 0 \end{matrix}\right) $ & $\left( \one^{1},X_1^{1} \right)$\\ \hline
        Fibonacci & $\mathrm{Id}$ & $\left( \one^{1},X_1^{1} \right)$ & $X_1 \mapsto \left( \begin{matrix}
             0 & 1 \\ 1 & 1 \end{matrix}\right) $ & $\left( \one^{1},X_1^{1} \right)$ \\ \hline
    \end{tabular}
    \caption{Modular invariants, NIM-reps and their respective exponents for rank 2 modular tensor categories}
    \label{MIrank2MTCs}
\end{table}

\begin{table}[H]
    \centering
    \begin{tabular}{c|c|c||c|c}
         \textbf{MTC} & \textbf{Modular invariant} & \textbf{Exponent} & \textbf{NIM-rep} & \textbf{Exponent} \\ \hline
       $\mathbb{Z}_3$ & $\mathrm{Id}$ & $\left( \one^{1},X_1^{1},X_2^{1} \right)$ & $X_1 \mapsto \begin{pmatrix}
           0 & 1 & 0 \\ 0 & 0 & 1 \\ 1 & 0 & 0 
       \end{pmatrix}$,  & $\left( \one^{1},X_1^{1},X_2^{1} \right)$ \\ &&&$X_2 \mapsto \begin{pmatrix}
           0 & 0 & 1 \\ 1 & 0 & 0 \\ 0 & 1 & 0
       \end{pmatrix}$& \\ & $\begin{pmatrix}
  1 & 0 & 0\\ 
  0 & 0 & 1\\
  0 & 1 & 0
\end{pmatrix}$ & $\left( \one^{1},X_1^{0},X_2^{0} \right)$ & $X_1, X_2 \mapsto \left( 1 \right)$ & $\left( \one^{1},X_1^{0},X_2^{0} \right)$ \\ \hline
      Ising/$A \left( 1,2\right)$
      & $\mathrm{Id}$ & $\left( \one^{1},X_1^{1},X_2^{1} \right)$ & $X_1 \mapsto \begin{pmatrix}
         0 & 1 & 0 \\ 1 & 0 & 0 \\ 0 & 0 & 1 
      \end{pmatrix}$,  & $\left( \one^{1},X_1^{1},X_2^{1} \right)$ \\ &&& $X_2 \mapsto \begin{pmatrix}
         0 & 0 & 1 \\ 0 & 0 & 1 \\ 1 & 1 & 0 
      \end{pmatrix}$& \\ \hline
    $A( 1,5)_{1/2}$ & $\mathrm{Id}$ &$\left( \one^{1},X_1^{1},X_2^{1} \right)$ & $X_1 \mapsto \begin{pmatrix}
0 & 1 & 0 \\ 1 & 0 & 1 \\ 0 & 1 & 1         
    \end{pmatrix},$ & $\left( \one^{1},X_1^{1},X_2^{1} \right)$ \\ &&&$X_2 \mapsto \begin{pmatrix}
0 & 0 & 1 \\ 0 & 1 & 1 \\ 1 & 1 & 1         
    \end{pmatrix}$& \\\hline
    \end{tabular}
\caption{Modular invariants, NIM-reps and their respective exponents for rank 3 modular tensor categories}
\label{MIrank3MTCs}
\end{table}

\begin{table}[H]
    \centering
    \begin{tabular}{c|c|c||c|c}
         \textbf{MTC} & \textbf{Modular invariant} & \textbf{Exponent} & \textbf{NIM-rep} & \textbf{Exponent} \\  \hline
        $\mathbb{Z}_4$ & $\mathrm{Id}$ & $\left( \one^{1},X_1^{1},X_2^{1},X_3^{1}\right)$ &$X_1 \mapsto \begin{pmatrix}
      0 & 1 & 0 & 0 \\ 1 & 0 & 0 & 0 \\ 0 & 0 & 0 & 1 \\ 0 & 0 & 1 & 0
  \end{pmatrix}$,& $\left( \one^{1},X_1^{1},X_2^{1},X_3^{1}\right)$ \\
 &&&$X_2 \mapsto \begin{pmatrix}
    0 & 0 & 1 & 0 \\ 0 & 0 & 0 & 1 \\ 0 & 1 & 0 & 0 \\ 1 & 0 & 0 & 0
 \end{pmatrix}$,& \\  
 &&&$X_3 \mapsto \begin{pmatrix}
     0 & 0 & 0 & 1 \\ 0 & 0 & 1 & 0 \\ 1 & 0 & 0 & 0 \\ 0 & 1 & 0 & 0
 \end{pmatrix}$& \\
        & $\begin{pmatrix}
  1 & 0 & 0 & 0\\ 
  0 & 1 & 0 & 0\\
  0 & 0 & 0 & 1\\
  0 & 0 & 1 & 0
  \end{pmatrix}$ & $\left( \one^{1},X_1^{1},X_2^{0},X_3^{0} \right)$ &$X_1 \mapsto \begin{pmatrix}
      1 & 0 \\ 0 & 1 
  \end{pmatrix}$,& $\left( \one^{1},X_1^{1},X_2^{0},X_3^{0} \right)$  \\ 
  &&& $X_2, X_3 \mapsto \begin{pmatrix}
      0 & 1 \\ 1 & 0
  \end{pmatrix}$&\\
&&&$X_1, X_2, X_3 \mapsto \left( 1\right),$& $\left( \one^{1},X_1^{0},X_2^{0},X_3^{0} \right)$ \\
  \hline
           $ A(1,7)_{1/2}$ & $\mathrm{Id}$ & $ \left( \one^{1},X_1^{1},X_2^{1},X_3^{1} \right)$ &$X_1\mapsto \begin{pmatrix}
        0 & 1 & 0 & 0 \\ 1 & 0 & 1 & 0 \\ 0 & 1 & 0 & 1 \\ 0 & 0 & 1 & 1
    \end{pmatrix}$,& $ \left( \one^{1},X_1^{1},X_2^{1},X_3^{1} \right)$ \\
    &  & &$X_2\mapsto \begin{pmatrix}
        0 & 0 & 1 & 0 \\ 0 & 1 & 0 & 1 \\ 1 & 0 & 1 & 1 \\ 0 & 1 & 1 & 1
    \end{pmatrix}$,&  \\
    &&&$X_3 \mapsto \begin{pmatrix}
        0 & 0 & 0 & 1 \\ 0 & 0 & 1 & 1 \\ 0 & 1 & 1 & 1 \\ 1 & 1 & 1 & 1
    \end{pmatrix}$& \\ \hline
\end{tabular}

\caption{Modular invariants, NIM-reps and their respective exponents for rank 4 modular tensor categories, $\mathbb{Z}_4$ and $ A(1,7)_{1/2}$ cases}
\label{MIrank4MTCs-1}
\end{table}

   \begin{table}[H]
    \centering
    \begin{tabular}{c|c|c||c|c}
         \textbf{MTC} & \textbf{Modular invariant} & \textbf{Exponent} & \textbf{NIM-rep} & \textbf{Exponent} \\  \hline 
    Toric  & $\mathrm{Id}$ & $\left( \one^{1},X_1^{1},X_2^{1},X_3^{1} \right)$ &$X_1 \mapsto \begin{pmatrix}
      0 & 1 & 0 & 0 \\ 1 & 0 & 0 & 0 \\ 0 & 0 & 0 & 1 \\ 0 & 0 & 1 & 0
  \end{pmatrix}$,& $\left( \one^{1},X_1^{1},X_2^{1},X_3^{1} \right)$ \\ 
  &&&$X_2 \mapsto \begin{pmatrix}
      0 & 0 & 1 & 0 \\ 0 & 0 & 0 & 1 \\ 1 & 0 & 0 & 0 \\ 0 & 1 & 0 & 0
  \end{pmatrix}$,&\\
  &&&$X_3 \mapsto \begin{pmatrix}
      0 & 0 & 0 & 1 \\ 0 & 0 & 1 & 0 \\ 0 & 1 & 0 & 0 \\ 1 & 0 & 0 & 0
  \end{pmatrix}$& \\
  & $\begin{pmatrix}
  1 & 1 & 0 & 0\\ 
  1 & 1 & 0 & 0\\
  0 & 0 & 0 & 0\\
  0 & 0 & 0 & 0
  \end{pmatrix}$ & $\left( \one^{1},X_1^{1},X_2^{0},X_3^{0} \right)$ &$X_1 \mapsto \begin{pmatrix}
      1 & 0 \\ 0 & 1 
  \end{pmatrix}$,& $\left( \one^{1},X_1^{1},X_2^{0},X_3^{0} \right)$ \\
  &&&$X_2,X_3 \mapsto \begin{pmatrix}
      0 & 1 \\ 1 & 0 
  \end{pmatrix}$& \\
   & $\begin{pmatrix}
  1 & 0 & 1 & 0\\ 
  0 & 0 & 0 & 0\\
  1 & 0 & 1 & 0\\
  0 & 0 & 0 & 0
  \end{pmatrix}$ & $\left( \one^{1},X_1^{0},X_2^{1},X_3^{0} \right)$ &$X_2 \mapsto \begin{pmatrix}
      1 & 0 \\ 0 & 1
  \end{pmatrix}$,& $\left( \one^{1},X_1^{0},X_2^{1},X_3^{0} \right)$\\
  &&&$X_1, X_3 \mapsto \begin{pmatrix}
      0 & 1 \\ 1 & 0
  \end{pmatrix}$& \\
    & $\begin{pmatrix}
  1 & 0 & 0 & 0\\ 
  0 & 0 & 1 & 0\\
  0 & 1 & 0 & 0\\
  0 & 0 & 0 & 1
  \end{pmatrix}$ & $\left( \one^{1},X_1^{0},X_2^{0},X_3^{1} \right)$ & $X_3 \mapsto \begin{pmatrix}
      1 & 0 \\ 0 & 1
  \end{pmatrix},$ & $\left( \one^{1},X_1^{0},X_2^{0},X_3^{1} \right)$ \\
  &&&$X_1, X_2 \mapsto \begin{pmatrix}
      0 & 1 \\ 1 & 0
  \end{pmatrix}$& \\
         & $\begin{pmatrix}
  1 & 1 & 0 & 0\\ 
  0 & 0 & 0 & 0\\
  1 & 1 & 0 & 0\\
  0 & 0 & 0 & 0
  \end{pmatrix}$ & $\left( \one^{1},X_1^{0},X_2^{0},X_3^{0} \right)$ &$X_1,X_2,X_3 \mapsto \left( 1 \right)$& $\left( \one^{1},X_1^{0},X_2^{0},X_3^{0} \right)$ \\ 
   & $\begin{pmatrix}
  1 & 0 & 1 & 0\\ 
  1 & 0 & 1 & 0\\
  0 & 0 & 0 & 0\\
  0 & 0 & 0 & 0
  \end{pmatrix}$ & $\left( \one^{1},X_1^{0},X_2^{0},X_3^{0}\right)$ & -- & --  \\
   \hline
\end{tabular}
\caption{Modular invariants, NIM-reps and their respective exponents for rank 4 modular tensor categories, toric case}
\label{MIrank4MTCs-2}
\end{table}

\begin{table}[H]
    \centering
    \begin{tabular}{c|c|c||c|c}
         \textbf{MTC} & \textbf{Modular invariant} & \textbf{Exponent} & \textbf{NIM-rep} & \textbf{Exponent} \\ \hline

  ($D_{4},1$) & Id & $\left(\one^{1},X_1^{1},X_2^{1},X_3^{1}\right)$ &$X_1 \mapsto \begin{pmatrix}
      0 & 1 & 0 & 0 \\ 1 & 0 & 0 & 0 \\ 0 & 0 & 0 & 1 \\ 0 & 0 & 1 & 0
  \end{pmatrix}$,& $\left( \one^{1},X_1^{1},X_2^{1},X_3^{1} \right)$ \\
 &&&$X_2 \mapsto \begin{pmatrix}
    0 & 0 & 1 & 0 \\ 0 & 0 & 0 & 1 \\ 0 & 1 & 0 & 0 \\ 1 & 0 & 0 & 0
 \end{pmatrix}$,& \\  
 &&&$X_3 \mapsto \begin{pmatrix}
     0 & 0 & 0 & 1 \\ 0 & 0 & 1 & 0 \\ 1 & 0 & 0 & 0 \\ 0 & 1 & 0 & 0
 \end{pmatrix}$&   \\
  & $\begin{pmatrix}
  1 & 0 & 0 & 0\\ 
  0 & 1 & 0 & 0\\
  0 & 0 & 0 & 1\\
  0 & 0 & 1 & 0
  \end{pmatrix}$ & $\left(\one^{1},X_1^{1},X_2^{0},X_3^{0}\right)$&$X_1 \mapsto \begin{pmatrix}
      1 & 0 \\ 0 & 1 
  \end{pmatrix}$,& $\left( \one^{1},X_1^{1},X_2^{0},X_3^{0} \right)$ \\
  &&&$X_2,X_3 \mapsto \begin{pmatrix}
      0 & 1 \\ 1 & 0 
  \end{pmatrix}$& \\ 
 & $\begin{pmatrix}
  1 & 0 & 0 & 0\\ 
  0 & 0 & 0 & 1\\
  0 & 0 & 1 & 0\\
  0 & 1 & 0 & 0
  \end{pmatrix}$ & ($\one^{1},X_1^{0}X_2^{1},X_3^{0}$)& $X_2 \mapsto \begin{pmatrix}
      1 & 0 \\ 0 & 1
  \end{pmatrix}$,& $\left( \one^{1},X_1^{0},X_2^{1},X_3^{0} \right)$\\
  &&&$X_1, X_3 \mapsto \begin{pmatrix}
      0 & 1 \\ 1 & 0
  \end{pmatrix}$& \\
   & $\begin{pmatrix}
  1 & 0 & 0 & 0\\ 
  0 & 0 & 1 & 0\\
  0 & 1 & 0 & 0\\
  0 & 0 & 0 & 1
  \end{pmatrix}$ & ($\one^{1},X_1^{0},X_2^{0},X_3^{1}$)&  $X_3 \mapsto \begin{pmatrix}
      1 & 0 \\ 0 & 1
  \end{pmatrix},$ & $\left( \one^{1},X_1^{0},X_2^{0},X_3^{1} \right)$ \\
  &&&$X_1, X_2 \mapsto \begin{pmatrix}
      0 & 1 \\ 1 & 0
  \end{pmatrix}$&  \\
  & $\begin{pmatrix}
 1 & 0 & 0 & 0\\ 
 0 & 0 & 0 & 1\\
 0 & 1 & 0 & 0\\
 0 & 0 & 1 & 0
 \end{pmatrix}$ & ($\one^{1},X_1^{0},X_2^{0},X_3^{0}$)&$X_1,X_2,X_3 \mapsto \left( 1 \right)$& $\left( \one^{1},X_1^{0},X_2^{0},X_3^{0} \right)$\\
 & $\begin{pmatrix}
 1 & 0 & 0 & 0\\ 
 0 & 0 & 1 & 0\\
 0 & 0 & 0 & 1\\
 0 & 1 & 0 & 0
 \end{pmatrix}$ & ($\one^{1},X_1^{0},X_2^{0},X_3^{0}$) & -- & -- \\ \hline
  
    \end{tabular}
      \caption{Modular invariants, NIM-reps and their respective exponents for rank 4 modular tensor categories, ($D_{4},1$) case}
\label{MIrank4MTCs-3}
\end{table}


\bibliography{biblio}
\bibliographystyle{amsrefs}


\end{document}